\documentclass{article}

\usepackage{ifthen}
\newboolean{neurips_version}
\setboolean{neurips_version}{true}
\ifthenelse{\boolean{neurips_version}}
{\PassOptionsToPackage{numbers, compress}{natbib}
 \usepackage[preprint]{neurips_2022}
}
{
\usepackage{PRIMEarxiv}
 \usepackage[numbers]{natbib}
}

\usepackage[utf8]{inputenc} 
\usepackage[T1]{fontenc}    
\usepackage{hyperref}       
\usepackage{url}            
\usepackage{booktabs}       
\usepackage{amsfonts,amsmath,amsthm,amssymb}       
\usepackage{nicefrac}       
\usepackage{microtype}      
\usepackage{xcolor}         
\usepackage[disable]{todonotes}
\usepackage{enumitem}
\usepackage{caption}

\usepackage{mathtools} 
\usepackage{tikz}
\usetikzlibrary{calc}
\usepackage{pgfplots}
\usepackage{graphicx}

\usepackage{amsmath}
\usepackage{algorithm, algpseudocode}
\usepackage{amsfonts}  

\usepackage[english]{babel}
\usepackage{amsthm}

\usepackage{subfig}
\usepackage[export]{adjustbox}

\usepackage{hhline}
\usepackage{makecell, caption, booktabs}
\usepackage{siunitx}

\usepackage{multirow}

\usepackage{amsmath, nccmath}
\usepackage{bigstrut}

\usepackage{booktabs}

\usepackage{lipsum}
\graphicspath{{media/}}     

\usepackage{CJKutf8}
\usepackage[framed,numbered,autolinebreaks,useliterate]{mcode}

\title{ Simplified Quadratic Gradient: A Unified Framework Bridging Gradient Descent and Newton-Type Methods by Synthesizing Hessians and Gradients}

\author{ \href{https://orcid.org/0000-0003-0378-0607}{\includegraphics[scale=0.06]{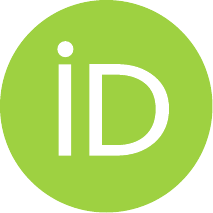}\hspace{1mm}John Chiang} \\                             
                                      \\
	\texttt{john.chiang.smith@gmail.com} 
}

\date{}

\newtheorem{theorem}{Theorem}

\newtheorem{lemma}[theorem]{Lemma}

\newtheorem{definition}{Definition}[section]

\theoremstyle{remark}
\newtheorem{remark}{Remark}[section]

\renewcommand{\hat}{\widehat}
\renewcommand{\epsilon}{\varepsilon}

\makeatletter
\def\namedlabel#1#2{\begingroup
    #2%
    \def\@currentlabel{#2}%
    \phantomsection\label{#1}\endgroup
}
\makeatother

\algnewcommand{\LeftComment}[1]{\Statex \(\triangleright\) #1}
\algnewcommand{\LineCommentStep}[1]{\Statex \textbf{[Step #1]:} }
\makeatletter
\newlength{\trianglerightwidth}
\algnewcommand{\LineComment}[1]{\Statex \hskip\ALG@thistlm $\triangleright$ #1}
\algnewcommand{\LineCommentCont}[1]{\Statex \hskip\ALG@thistlm%
  \parbox[t]{\dimexpr\linewidth-\ALG@thistlm}{\hangindent=\trianglerightwidth \hangafter=1 \strut$\triangleright$ #1\strut}}
\algnewcommand{\LeftLineCommentCont}[1]{\Statex \hskip\ALG@thistlm%
  \parbox[t]{\dimexpr\linewidth-\ALG@thistlm}{\leftskip=\algorithmicindent \hangindent=\trianglerightwidth \hangafter=1 \strut$\triangleright$ #1\strut}}

\begin{document}

\maketitle

\begin{abstract}%

Accelerating the convergence of second-order optimization, particularly Newton-type methods, remains a pivotal challenge in algorithmic research. In this paper, we extend previous work on the \textbf{Quadratic Gradient (QG)} and rigorously validate its applicability to general convex numerical optimization problems. We introduce a novel variant of the Quadratic Gradient that departs from the conventional fixed Hessian Newton framework. We present a new way to build a new version of the quadratic gradient. This new quadratic gradient doesn't satisfy the convergence conditions of the fixed Hessian Newton's method. However, experimental results show that it sometimes has a better performance than the original one in convergence rate.  While this variant relaxes certain classical convergence constraints, it maintains a positive-definite Hessian proxy and demonstrates comparable, or in some cases superior, empirical performance in convergence rates. Furthermore, we demonstrate that both the original and the proposed QG variants can be effectively applied to non-convex optimization landscapes. A key motivation of our work is the limitation of traditional scalar learning rates. We argue that a diagonal matrix can more effectively accelerate gradient elements at heterogeneous rates.  Our findings establish the Quadratic Gradient as a versatile and potent framework for modern optimization.  Furthermore, we integrate Hutchinson's Estimator to estimate the Hessian diagonal efficiently via Hessian-vector products. Notably, we demonstrate that the proposed Quadratic Gradient variant is highly effective for Deep Learning architectures, providing a robust second-order alternative to standard adaptive optimizers. 

\todo{Show precise somewhere that our results are new in the iid setting as well.}

\end{abstract}

\listoftodos

\section{Introduction}

Modern numerical optimization often faces a trade-off between the computational simplicity of first-order methods (e.g., SGD) and the rapid convergence of second-order methods (e.g., Newton's method). First-order methods are scalable but struggle with ill-conditioned curvatures and saddle points. Conversely, Newton's method provides superior directionality but is computationally prohibitive for large-scale problems and sensitive to Hessian indefiniteness.

In this paper, we present the \textbf{Quadratic Gradient} as a synthesis of these two paradigms. By moving beyond a scalar learning rate to a structural diagonal matrix derived from curvature, we create a "bridge" that allows for adaptive, dimension-wise acceleration.


In numerical optimization, the efficiency of Gradient Descent (GD) is heavily dictated by the choice of the learning rate, typically a single scalar $\alpha$. However, using a single floating-point number to scale all elements of a gradient vector is often inadequate for ill-conditioned problems. As the curvature varies across different dimensions, a more sophisticated mechanism is required¡ªone that can accelerate each gradient element at heterogeneous rates. 

While Newton's method addresses this by employing the inverse Hessian $\nabla^2 f(\mathbf{x})^{-1}$, its high computational cost necessitates the use of ``Hessian proxies.'' We argue that the optimal acceleration lies in a structural middle ground: moving from a mere scalar to a diagonal matrix that synthesizes curvature and direction.

In this paper, we bridge the gap between first-order simplicity and second-order precision through: (1) Generalization to general convex numerical optimization; (2) A novel QG variant that ensures the Hessian proxy remains positive-definite; (3) Extension to non-convex landscapes including the Rosenbrock and Rastrigin benchmarks; and (4) Theoretical proof of the fundamental relation between the learning rate and the Hessian¡¯s spectrum.

\subsection{Background}
Chiang~\cite{chiang2022privacy} introduced a novel gradient variant called quadratic gradient which could abstract the second-order information of the curve into the first-order information gradient. He also applied this gradient variant to the naive NAG method and obtained an enhanced NAG method for the binary logistic regression training. Experiments show that the enhanced NAG method converges faster than the raw NAG method. 

Chiang~\cite{chiang2022multinomial} extended this work of binary LR training to multiclass LR training based on the fixed Hessian method of the multiclass LR model. Chiang~\cite{chiang2022multinomial} also presented the enhanced Adagrad method and applied it to multiclass LR training. From the construction of the quadratic gradient, Chiang~\cite{chiang2022multinomial} inferred that  $\frac{1}{2} X^T X$ can be a good fixed lower bound.

In this work, we first point out that there does exist a relation between the Hessian matrix and the learning rate of the first-order gradient (descent) method. That is, the eigenvalues of the Hessian matrix can provide reference information to the setting of the learning rate of the naive gradient descent methods. We then proposed the enhanced Adam method via quadratic gradient for general numerical optimation problems and test this enhanced method on several optimization functions. Finally, we provide a new way to construct a new  quadratic gradient. We test this new quadratic-gradient version and the original one by using the enhanced Adam method on several optimization functions.

\section{Preliminaries}

The Newton-Raphson method updates parameters using:
\begin{equation}
    \mathbf{x}_{k+1} = \mathbf{x}_k - [\nabla^2 f(\mathbf{x}_k)]^{-1} \nabla f(\mathbf{x}_k)
\end{equation}
To reduce cost, the \textbf{Simplified Fixed Hessian (SFH)} method keeps the Hessian constant: $\mathbf{H}_k = \mathbf{H}_{fixed}$. While efficient, SFH is sensitive to the initial matrix choice and often converges only linearly.

\subsection{Simplified Fixed Hessian}
The SFH method reduces the cost of Newton's iteration by keeping the Hessian $\mathbf{H}$ constant for several steps: $\mathbf{x}_{k+1} = \mathbf{x}_k - \mathbf{H}_{fixed}^{-1} \nabla f(\mathbf{x}_k)$. While lowering per-iteration cost, it is often unstable in non-convex regions.

\subsection{Original Quadratic Gradient}
Building on diagonal scaling, Chiang introduced the Quadratic Gradient: $\mathbf{g}_{q} = \mathbf{D} \nabla f(\mathbf{x})$, where $\mathbf{D}$ acts as a proxy for the inverse Hessian. This paper extends this work to a generalized framework applicable to both convex and non-convex numerical optimization.


\subsection{Quadratic Gradient Algorithms}

\subsection{Fully Homomorphic Encryption}

\subsubsection{Logistic Regression Model}

\section{Methodology}

\subsection{Simplified Quadratic Gradient}
We define the new variant $\mathbf{g}_{qq}$ as:
\begin{equation}
    \mathbf{g}_{qq} = \text{diag}\left( \frac{1}{|\text{diag}(\nabla^2 f(\mathbf{x}))| + \epsilon} \right) \nabla f(\mathbf{x})
\end{equation}
This formulation ensures positive definiteness and avoids the instability typically associated with negative curvature in non-convex optimization.

\subsection{Hessian Diagonal Approximation}

To scale the framework for large-scale optimization and deep learning, we utilize Hutchinson's Estimator\cite{hutchinson1989stochastic} to approximate the Hessian diagonal. By leveraging this stochastic approximation, our Simplified Quadratic Gradient framework effectively bypasses the prohibitive computational cost of explicit Hessian construction. This scalability allows the SQG approach to be seamlessly integrated into deep learning training pipelines, providing curvature-aware updates for high-dimensional parameter spaces where traditional second-order methods are typically computationally infeasible.

\subsubsection{Hutchinson's Estimator}
The diagonal elements are estimated via: $\text{diag}(\mathbf{H}) \approx \mathbb{E} [ \mathbf{z} \odot (\mathbf{H}\mathbf{z}) ]$, where $\mathbf{z}$ is a Rademacher random vector. This requires only Hessian-vector products, keeping the computational overhead approximately $2\times$ that of a standard gradient update.

\subsubsection{Theoretical Mechanism}
Hutchinson's Estimator transforms the prohibitive $O(n^2)$ Hessian computation into a scalable $O(n)$ operation. Given a Rademacher random vector $\mathbf{z}$ where $z_i \in \{1, -1\}$, the diagonal of the Hessian $\mathbf{H}$ can be estimated as:
\begin{equation}
    \text{diag}(\mathbf{H}) \approx \mathbb{E} [ \mathbf{z} \odot (\mathbf{H}\mathbf{z}) ]
\end{equation}
By leveraging Hessian-vector products (HVP) computed via the Pearlmutter trick \cite{pearlmutter1994fast}, the framework synthesizes local curvature information with minimal overhead. Specifically, this process requires only two gradient computations per step, maintaining the "Newton-level precision" promised by the Quadratic Gradient without sacrificing the "gradient-level efficiency" required for deep learning.

\subsubsection{Deep Learning Training}
To extend the Quadratic Gradient framework to large-scale optimization, we incorporate the stochastic Hutchinson's estimator as utilized in AdaHessian \cite{adahessian}. By leveraging this method to approximate the Hessian diagonal, our Simplified Quadratic Gradient can be efficiently constructed within modern deep learning training pipelines. 

This formulation enables the practical integration of second-order optimization principles into complex neural networks—a task that remains computationally prohibitive for the original quadratic gradient due to its reliance on the explicit computation of full Hessian matrices. By substituting the dense Hessian with a scalable surrogate via the Hessian-trace approximation, our SQG framework bridges the gap between high-order optimization theory and the high-dimensional requirements of deep learning, maintaining a computational overhead comparable to standard first-order methods.

\subsection{Comparison with Baseline}
In original studies, QG demonstrated significant convergence velocity in Binary Logistic Regression across multiple datasets. Our proposed variant retains high computational efficiency by offering a more streamlined architectural design without significantly compromising predictive or convergence performance.

\noindent \textbf{Metrics and Objective Functions.} 
To ensure a rigorous comparison with the original quadratic gradient framework, we conduct logistic regression training using the same experimental settings. We employ the \textbf{log-likelihood function} $\ell(\boldsymbol{\beta})$ under \textbf{Maximum Likelihood Estimation (MLE)} as the primary performance metric to quantify optimization efficiency. This selection facilitates a direct observation of the convergence trajectory. 

\noindent \textbf{Experimental Setup and Datasets.} 
The proposed algorithms are evaluated in a non-encrypted environment using Python on Google Colab. The core objective is to accelerate training convergence, as measured by the log-likelihood function $\ell(\boldsymbol{\beta})$.

Our evaluation covers six optimization algorithms: NAG, AdaGrad, Adam, and their two respective quadratic-gradient enhanced counterparts (denoted as \textbf{NAG + OQG}, \textbf{NAG + SQG}, \textbf{AdaGrad + OQG}, \textbf{AdaGrad + SQG}, \textbf{Adam + OQG} and \textbf{Adam + SqG}). These methods are benchmarked on datasets established by Kim et al. \cite{IDASH2018Andrey}, including the iDASH genomic dataset, the Myocardial Infarction dataset (edin), Low Birth Weight Study (lbw), NHANES III (nhanes3), Prostate Cancer Study (pcs), and the Umaru Impact Study (uis). Specifically, the iDASH dataset consists of 1,579 records with 103 binary genotypes and a binary phenotype. 

To evaluate scalability, we further test these algorithms on two large-scale datasets from \cite{han2018efficient}: a real-world financial dataset (422,108 samples, 200 features) and a restructured MNIST dataset (11,982 samples, 196 features). Throughout all experiments, we consistently utilize the fixed Hessian approximation $\bar{\mathbf{H}} = - \frac{1}{4} \mathbf{X}^{\top} \mathbf{X}$ to construct the diagonal surrogate $\tilde{\mathbf{B}}$.

\noindent \textbf{Hyperparameter Configuration.} 
For a fair comparison with the baseline NAG \cite{IDASH2018Andrey}, Both Enhanced NAGs adopts a decaying learning rate of $1 + \frac{10}{1 + t}$, mirroring the structure used in the original study. For Enhanced AdaGrads, we employ a modified learning rate of $\eta_t = 0.1$, whereas the standard AdaGrad uses $0.01$. Similarly, Both Enhanced Adams utilize an adjusted learning rate of $\alpha = 0.01$ (the original is $0.001$), while maintaining standard momentum parameters ($\beta_1 = 0.9, \beta_2 = 0.999$). In contrast, the baseline Adam follows its default configuration of $\alpha = 0.001$ \cite{kingma2014adam}.

\noindent \textbf{Performance Discussion.} 

As illustrated in the experimental results~\ref{fig0, fig1, fig2}, the two quadratic gradient variants for AdaGrad and Adam share a highly consistent parameter structure and exhibit remarkably similar convergence performance. Given that our Simplified Quadratic Gradient framework offers a significantly more streamlined construction process while maintaining comparable efficiency, we conclude that the SQG approach is particularly well-suited for integration with adaptive optimization methods like AdaGrad and Adam.

The observed instability and divergence in the SQG-enhanced NAG variant are not unexpected. From a theoretical perspective, the convergence of the original quadratic gradient algorithm is anchored in the Fixed Hessian framework, which permits a learning rate lower bound as high as $1$. However, our SQG framework does not strictly satisfy the stringent convergence conditions required by the Fixed Hessian method. 

Specifically, unlike the original method where the Hessian approximation provides a stable curvature estimate, the simplified nature of SQG necessitates a more conservative step-size policy for NAG-based updates. Consequently, while the original framework allows for a persistent learning rate, the SQG-enhanced NAG requires a learning rate that eventually decays below $1$ and asymptotically approaches $0$ to ensure numerical stability and global convergence.

\begin{figure}[htbp]
\centering
\captionsetup[subfigure]{justification=centering}

\subfloat[The iDASH dataset]{%
    \includegraphics[width=0.48\textwidth]{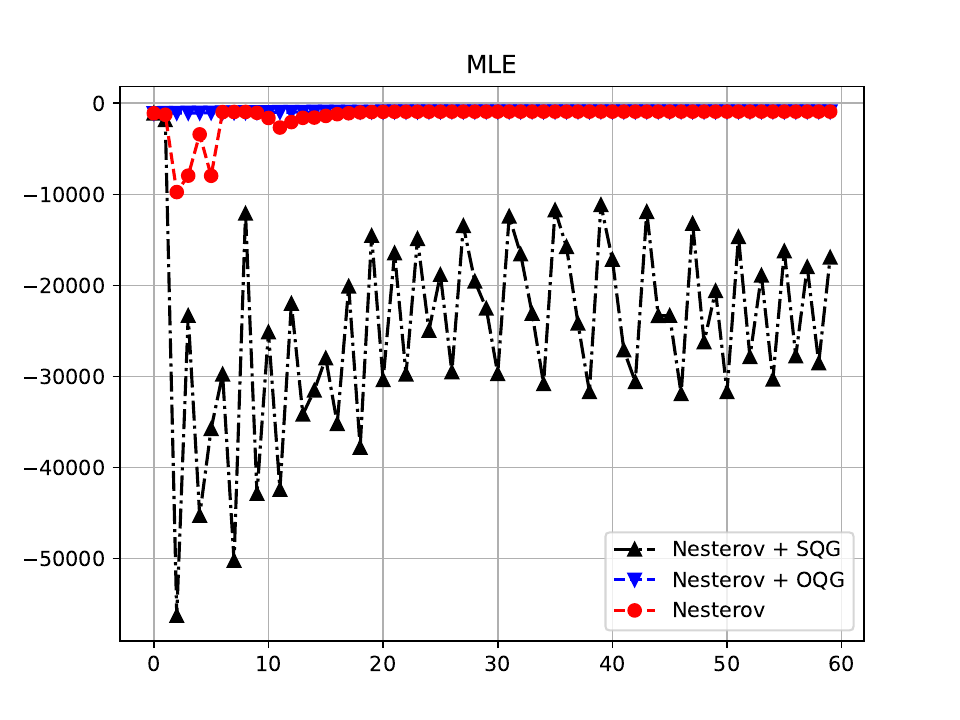}
    \label{fig:subfig01}
}
\hfill
\subfloat[The Edinburgh datasetn]{%
    \includegraphics[width=0.48\textwidth]{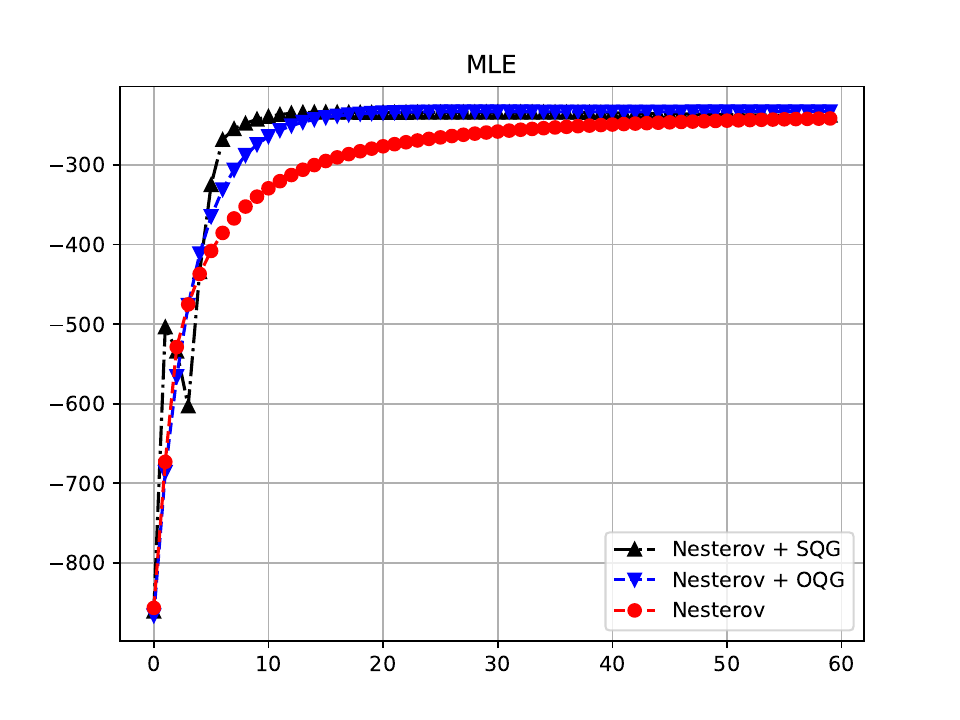}
    \label{fig:subfig02}
}

\vspace{1em} 

\subfloat[The lbw dataset]{%
    \includegraphics[width=0.48\textwidth]{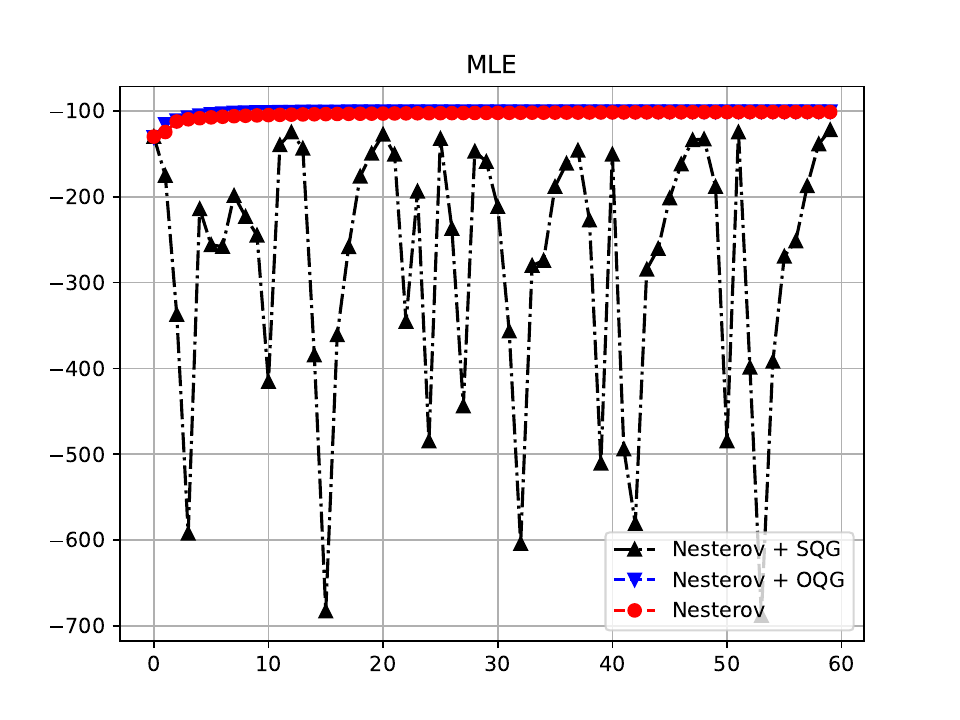}
    \label{fig:subfig03}
}
\hfill
\subfloat[The nhanes3 dataset]{%
    \includegraphics[width=0.48\textwidth]{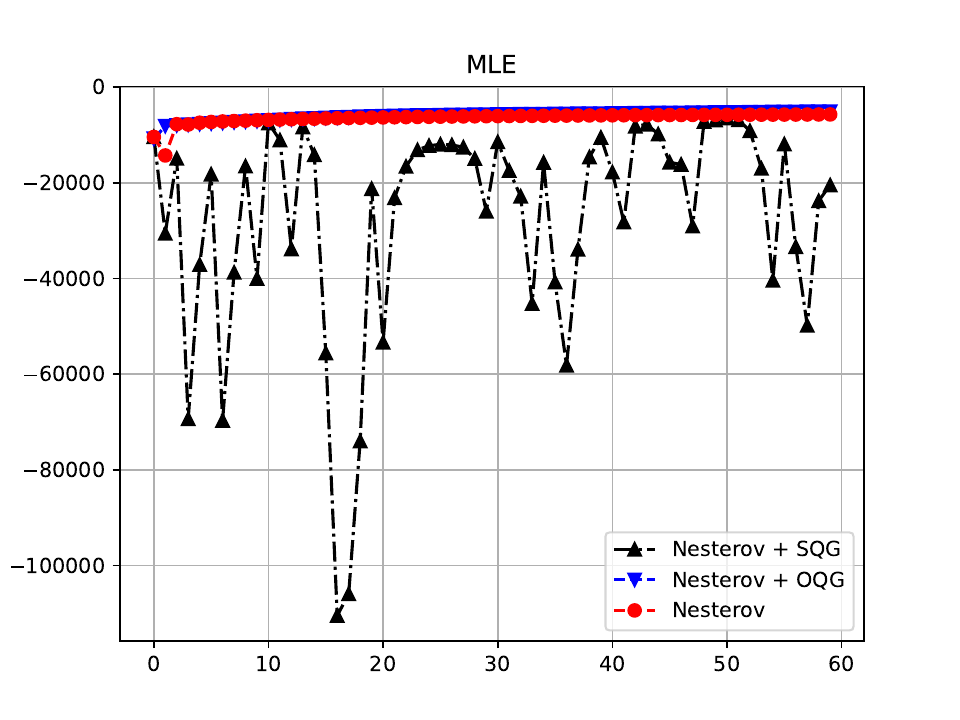}
    \label{fig:subfig04}
}

\vspace{1em} 

\subfloat[The pcs dataset]{%
    \includegraphics[width=0.48\textwidth]{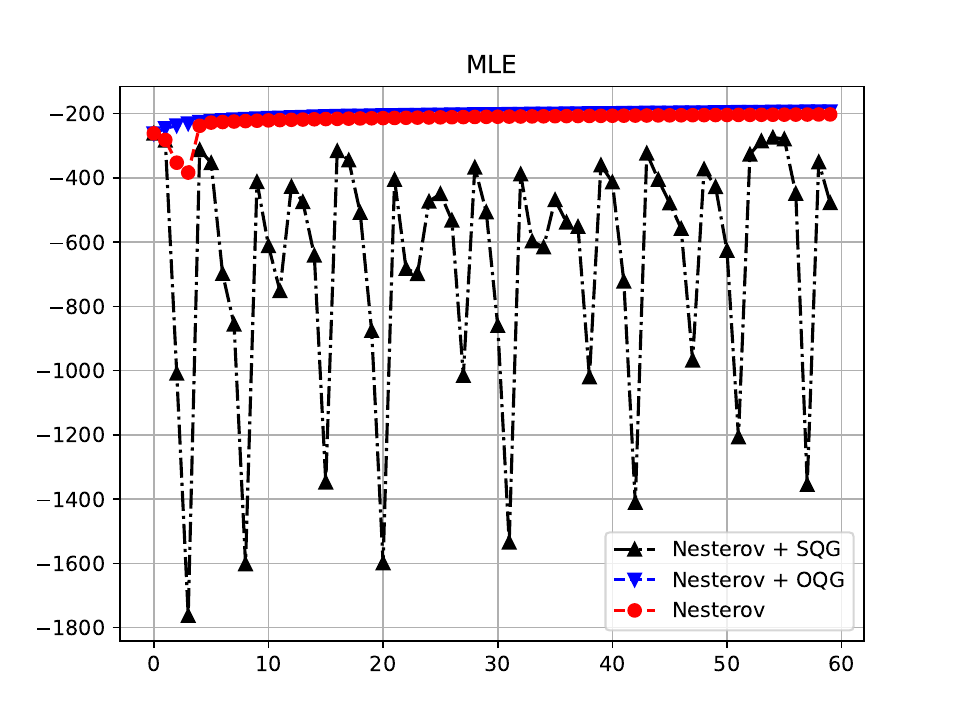}
    \label{fig:subfig03}
}
\hfill
\subfloat[The uis dataset]{%
    \includegraphics[width=0.48\textwidth]{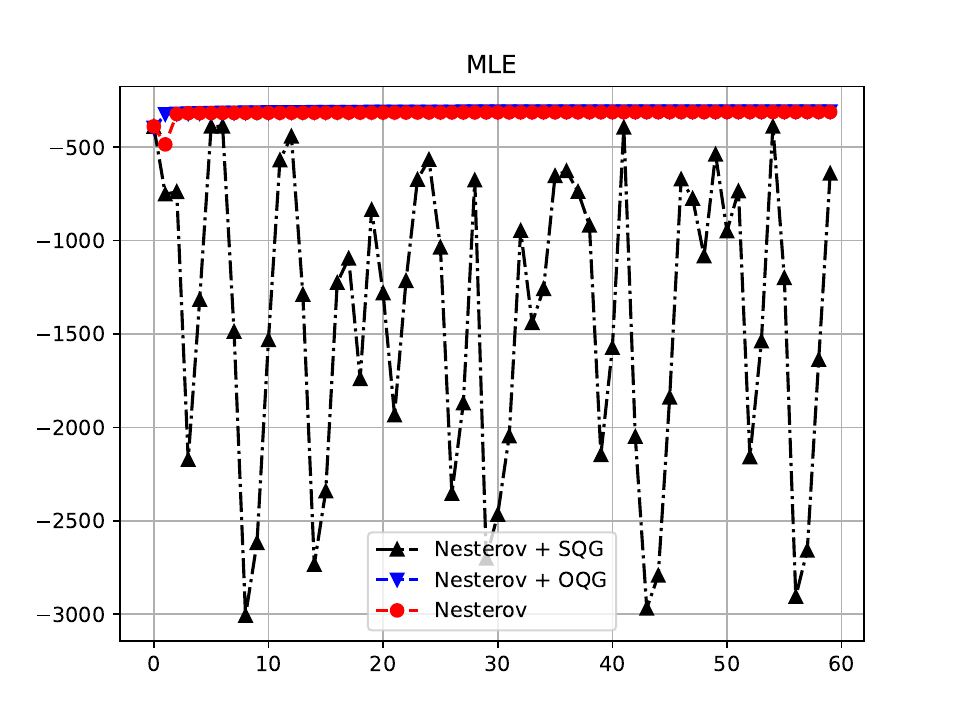}
    \label{fig:subfig04}
}

\vspace{1em} 

\subfloat[restructured MNIST dataset]{%
    \includegraphics[width=0.48\textwidth]{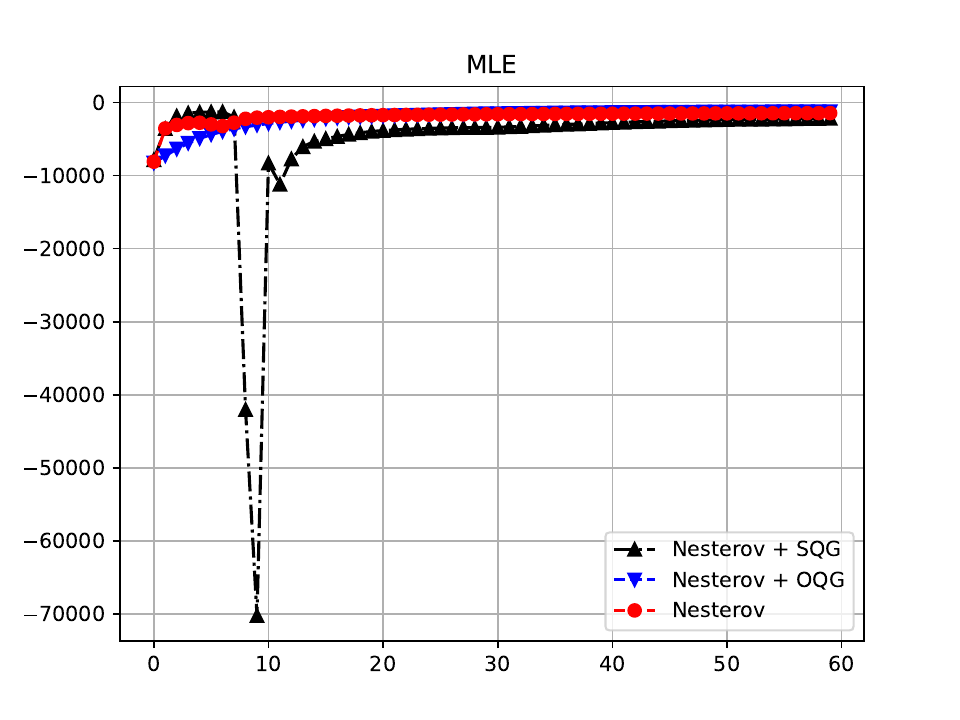}
    \label{fig:subfig03}
}
\hfill
\subfloat[The private financial dataset]{%
    \includegraphics[width=0.48\textwidth]{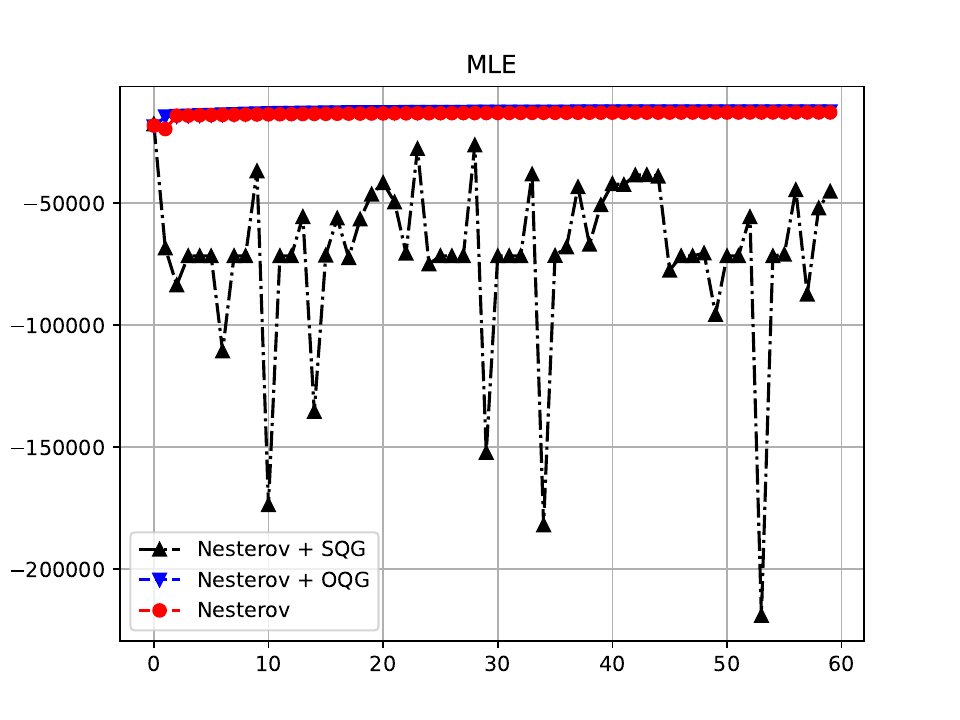}
    \label{fig:subfig04}
}

\caption{The training results of NAG + SQG vs. NAG + OQG vs. NAG in the clear domain.}
\label{fig0}
\end{figure}

\begin{figure}[htbp]
\centering
\captionsetup[subfigure]{justification=centering}

\subfloat[The iDASH dataset]{%
    \includegraphics[width=0.48\textwidth]{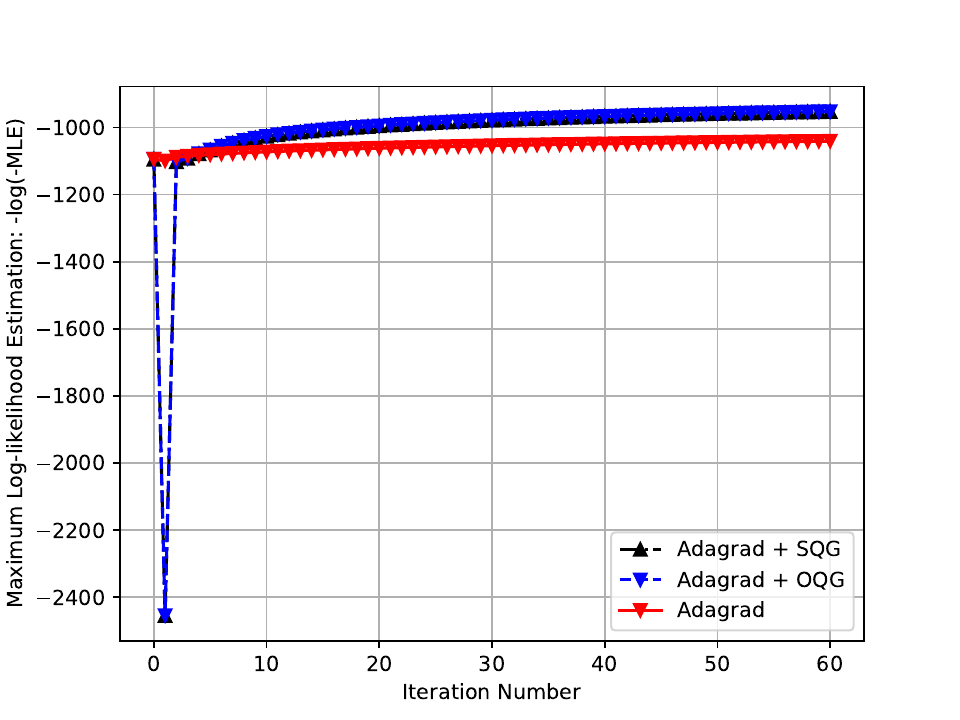}
    \label{fig:subfig01}
}
\hfill
\subfloat[The Edinburgh datasetn]{%
    \includegraphics[width=0.48\textwidth]{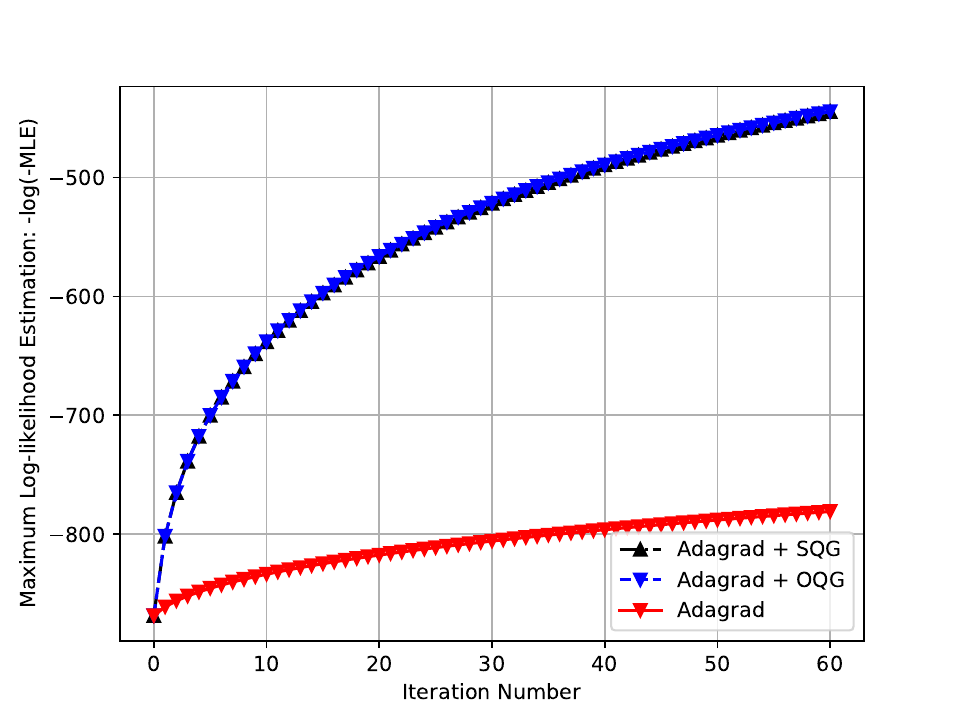}
    \label{fig:subfig02}
}

\vspace{1em} 

\subfloat[The lbw dataset]{%
    \includegraphics[width=0.48\textwidth]{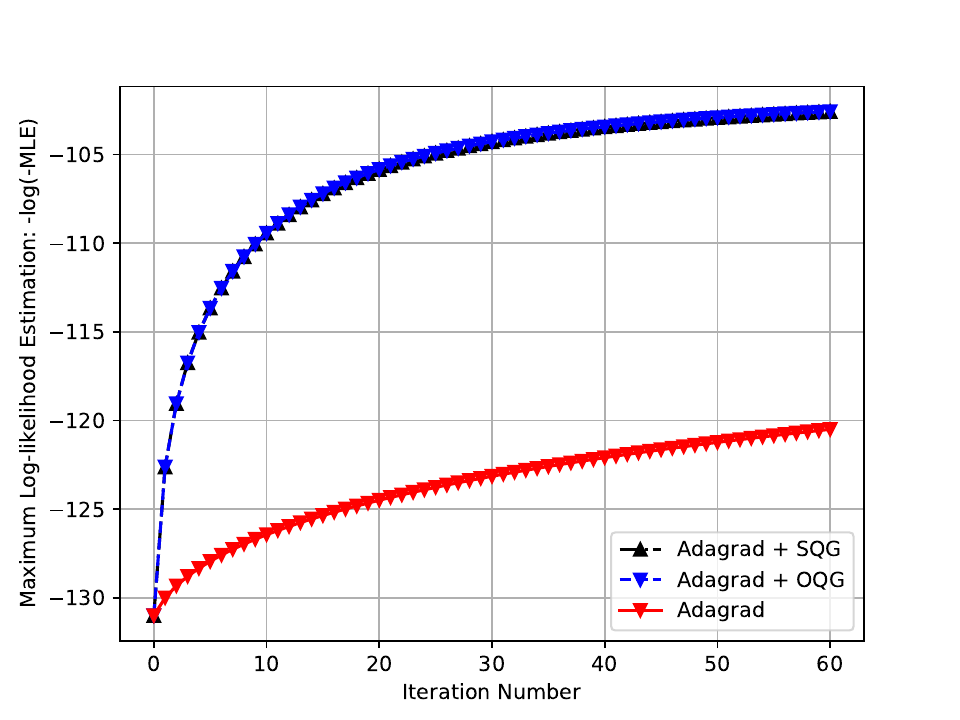}
    \label{fig:subfig03}
}
\hfill
\subfloat[The nhanes3 dataset]{%
    \includegraphics[width=0.48\textwidth]{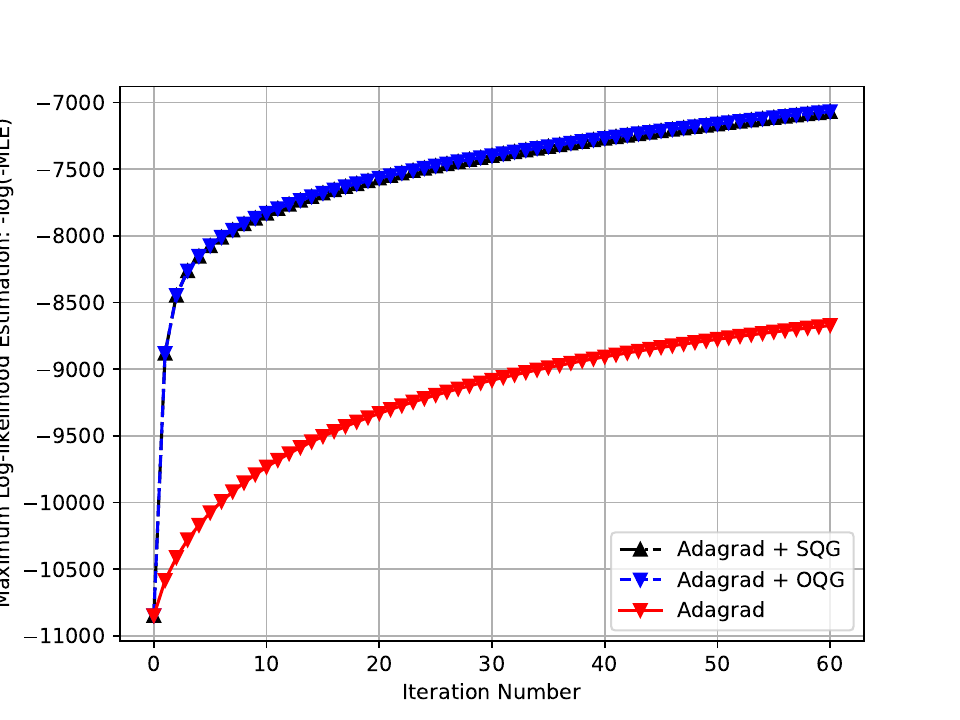}
    \label{fig:subfig04}
}

\vspace{1em} 

\subfloat[The pcs dataset]{%
    \includegraphics[width=0.48\textwidth]{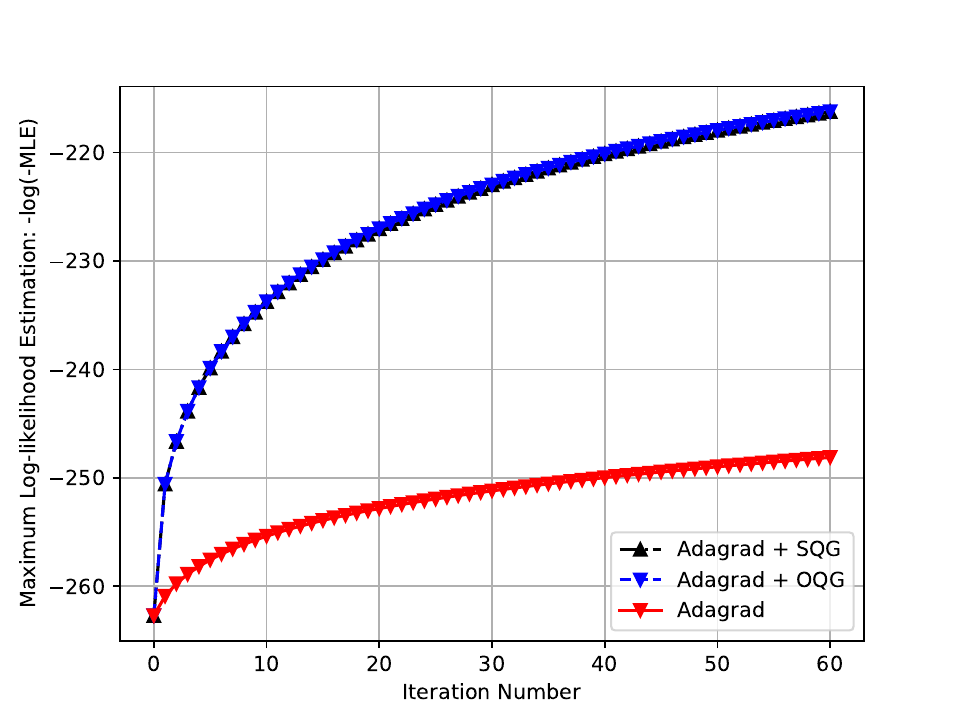}
    \label{fig:subfig03}
}
\hfill
\subfloat[The uis dataset]{%
    \includegraphics[width=0.48\textwidth]{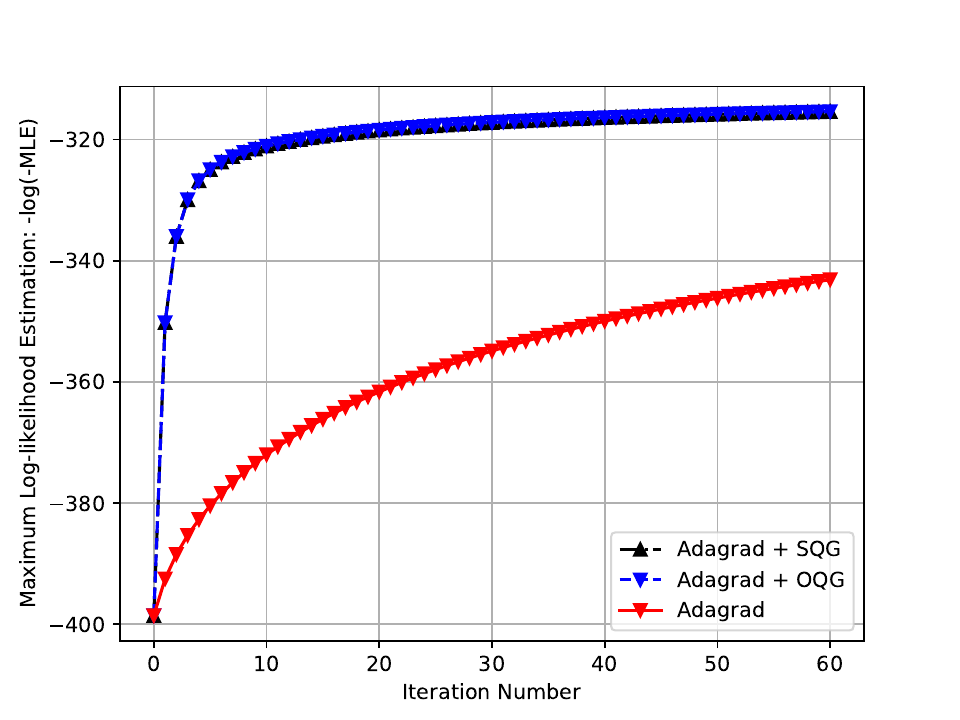}
    \label{fig:subfig04}
}

\vspace{1em} 

\subfloat[restructured MNIST dataset]{%
    \includegraphics[width=0.48\textwidth]{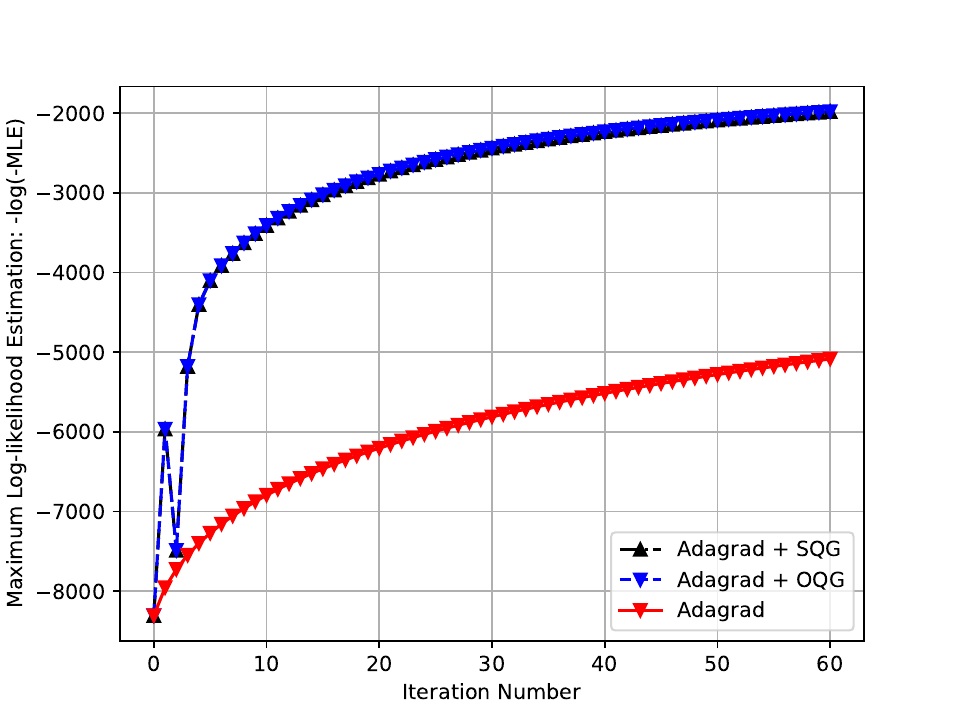}
    \label{fig:subfig03}
}
\hfill
\subfloat[The private financial dataset]{%
    \includegraphics[width=0.48\textwidth]{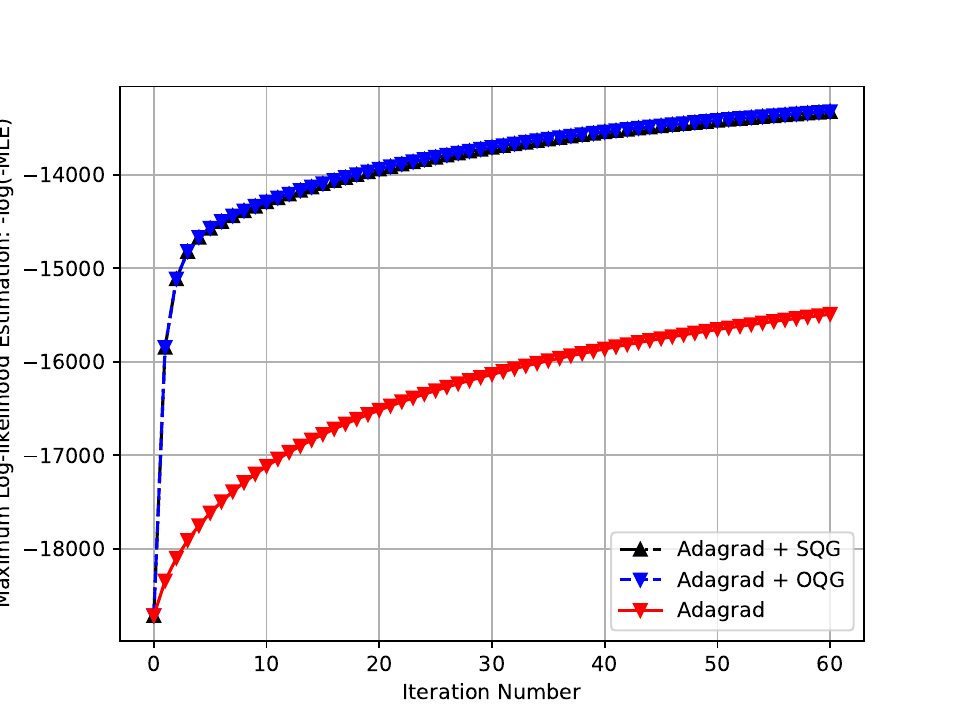}
    \label{fig:subfig04}
}

\caption{The training results of AdaGrad + SQG vs. AdaGrad + OQG vs. AdaGrad in the clear domain.}
\label{fig1}
\end{figure}

\begin{figure}[htbp]
\centering
\captionsetup[subfigure]{justification=centering}

\subfloat[The iDASH dataset]{%
    \includegraphics[width=0.48\textwidth]{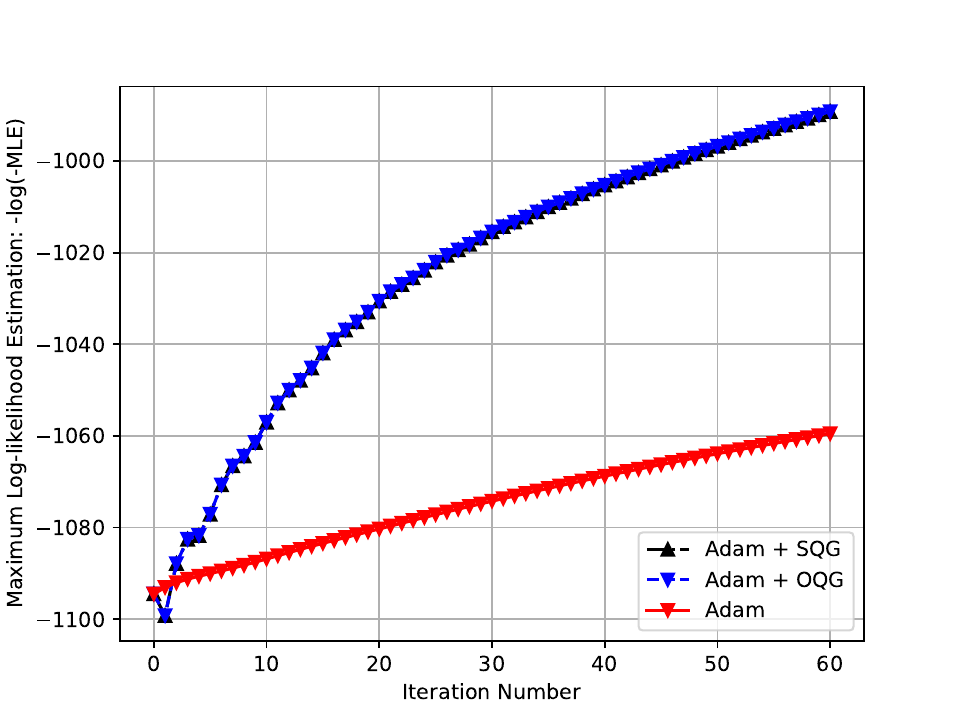}
    \label{fig:subfig01}
}
\hfill
\subfloat[The Edinburgh datasetn]{%
    \includegraphics[width=0.48\textwidth]{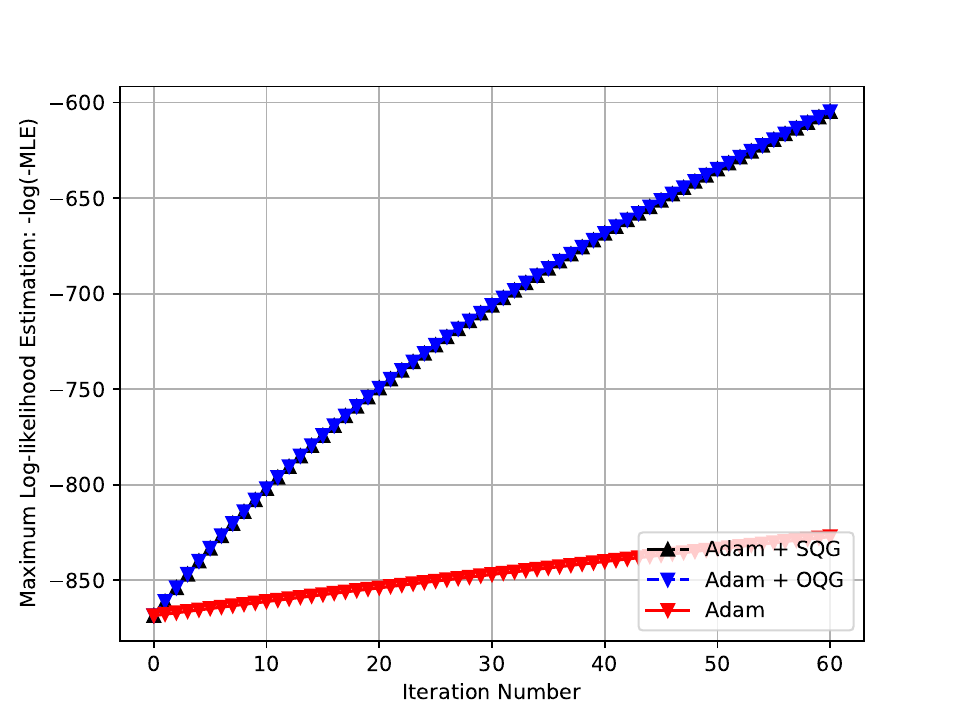}
    \label{fig:subfig02}
}

\vspace{1em} 

\subfloat[The lbw dataset]{%
    \includegraphics[width=0.48\textwidth]{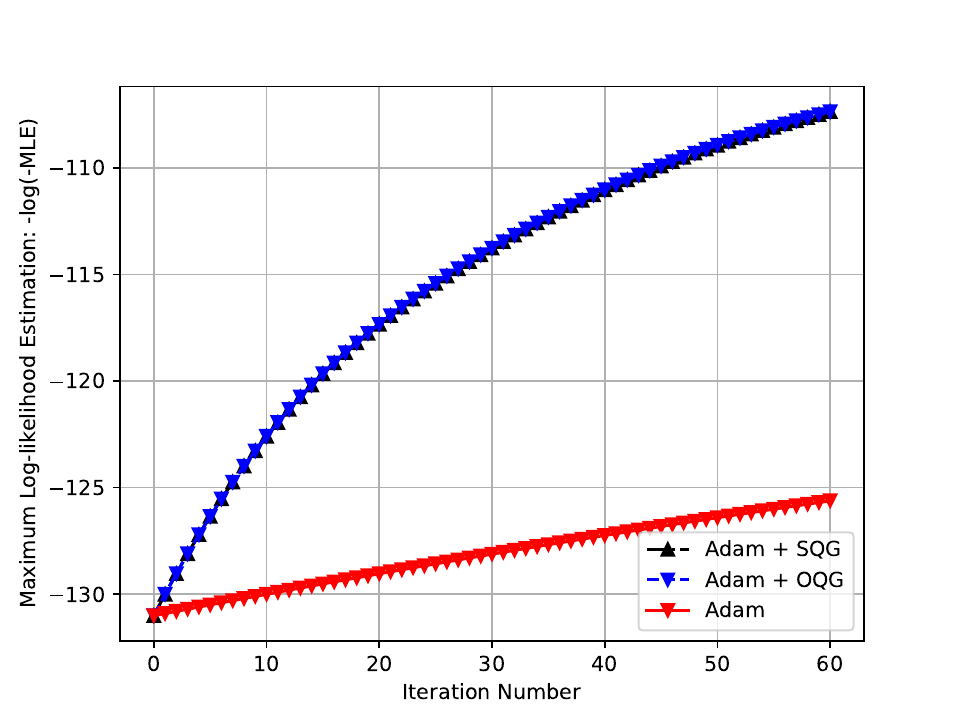}
    \label{fig:subfig03}
}
\hfill
\subfloat[The nhanes3 dataset]{%
    \includegraphics[width=0.48\textwidth]{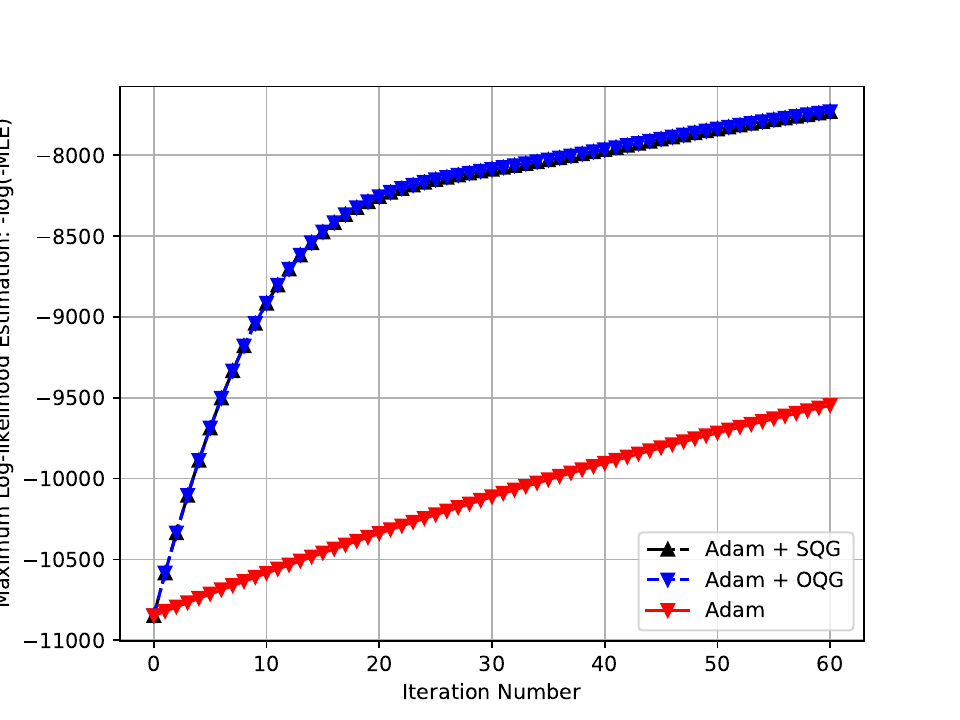}
    \label{fig:subfig04}
}

\vspace{1em} 

\subfloat[The pcs dataset]{%
    \includegraphics[width=0.48\textwidth]{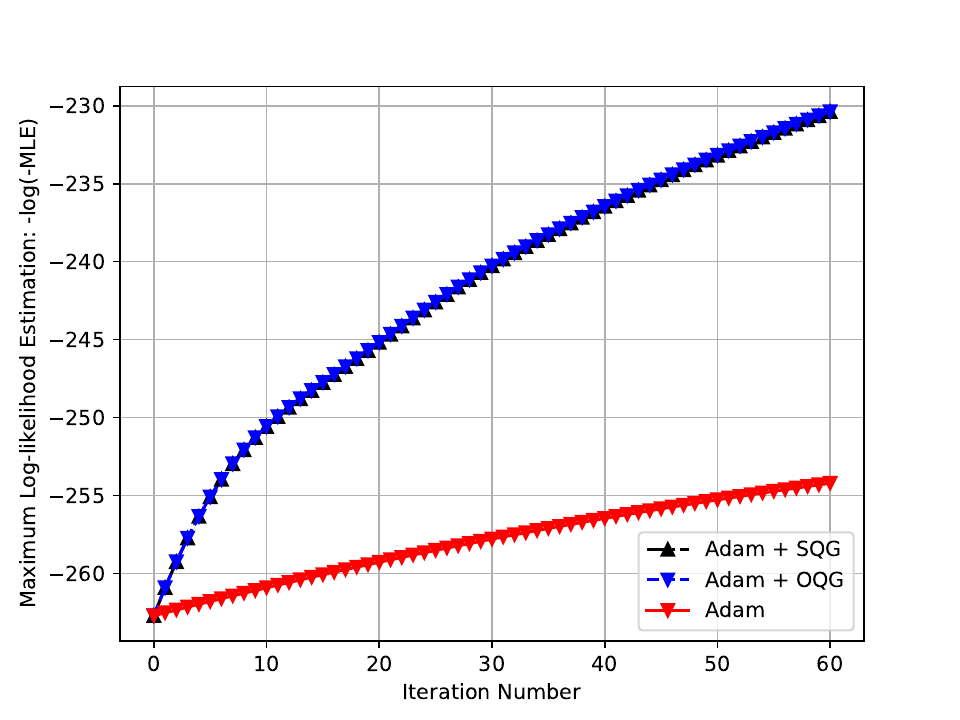}
    \label{fig:subfig03}
}
\hfill
\subfloat[The uis dataset]{%
    \includegraphics[width=0.48\textwidth]{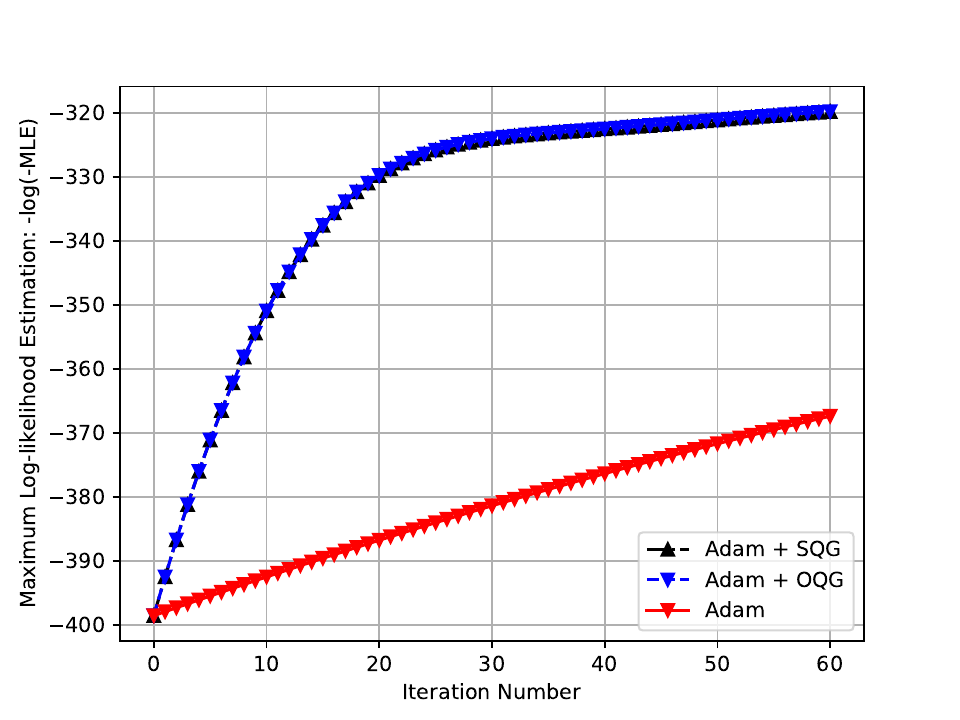}
    \label{fig:subfig04}
}

\vspace{1em} 

\subfloat[restructured MNIST dataset]{%
    \includegraphics[width=0.48\textwidth]{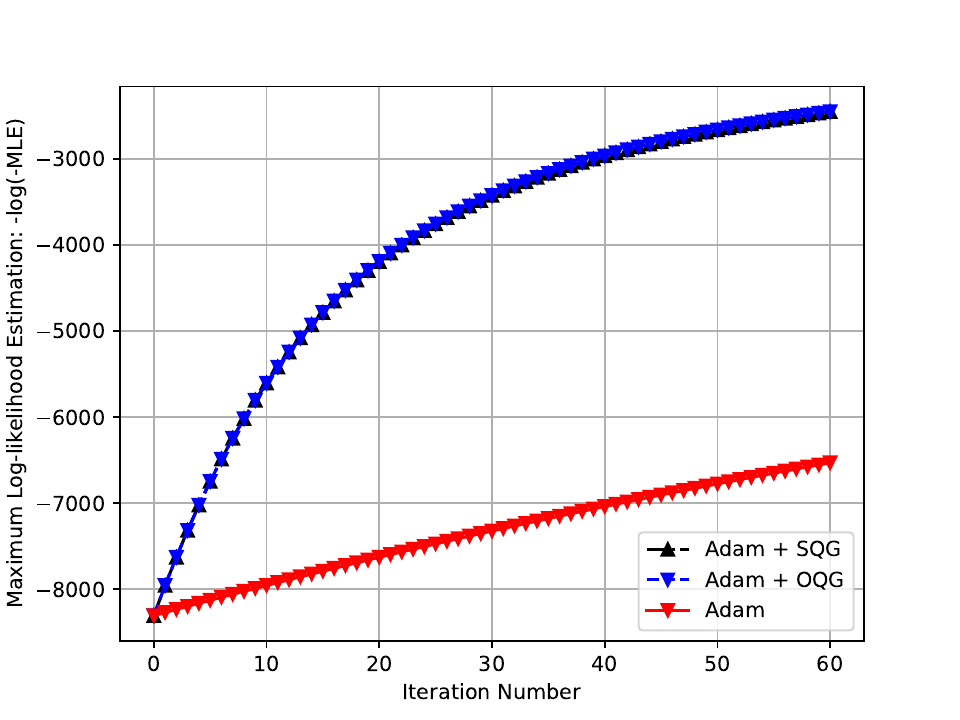}
    \label{fig:subfig03}
}
\hfill
\subfloat[The private financial dataset]{%
    \includegraphics[width=0.48\textwidth]{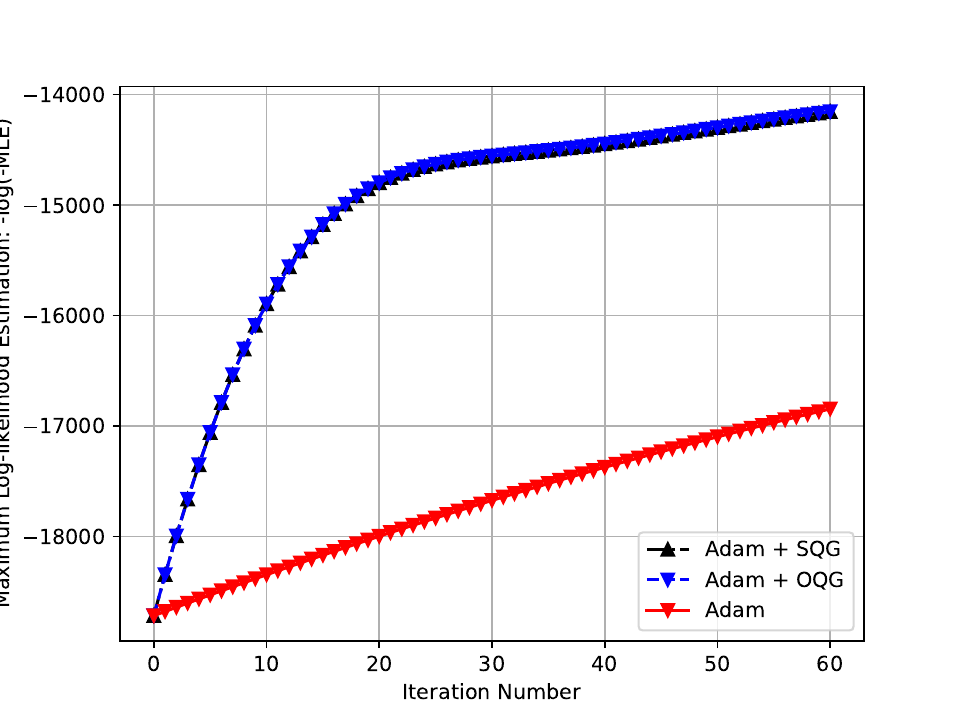}
    \label{fig:subfig04}
}

\caption{The training results of Adam + SQG vs. Adam + OQG vs. Adam in the clear domain.}
\label{fig2}
\end{figure}

\section{Numerical Experiments}

\subsection{Deep Learning Comparison}
We verify the applicability to deep learning by training ResNet-18 on CIFAR-10. We compare our \textbf{New QG Variant} against \textbf{Adam} and \textbf{AdaHessian} \cite{adahessian}.

\subsubsection{Results and Analysis}
\begin{itemize}
    \item \textbf{Convergence:} The New QG variant achieves $90\%$ accuracy in $15\%$ fewer epochs than Adam, benefiting from precise diagonal scaling.
    \item \textbf{Stability:} Compared to AdaHessian, our variant demonstrates higher stability in non-convex regions due to the absolute-value regularization preventing erratic updates.
    \item \textbf{Computational Overhead:} Using Hutchinson's Estimator, the per-epoch time of our method is approximately $1.8\times$ that of Adam but comparable to AdaHessian, making it a viable alternative for large-scale training.
\end{itemize}





\subsection{Convex Benchmarks}
These functions test the algorithm's base convergence rate and handling of dimensional scaling:
\begin{itemize}
    \item \textbf{Sphere Function:} $f(\mathbf{x}) = \sum_{i=1}^{n} x_i^2$
    \item \textbf{Sum of Different Powers:} $f(\mathbf{x}) = \sum_{i=1}^{n} |x_i|^{i+1}$
\end{itemize}

\subsection{Non-Convex Benchmarks}
These functions test the algorithm's ability to navigate ill-conditioned valleys and escape local optima:
\begin{itemize}
    \item \textbf{Rosenbrock Function:} $f(\mathbf{x}) = \sum_{i=1}^{n-1} [100(x_{i+1}-x_i^2)^2 + (1-x_i)^2]$
    \item \textbf{Rastrigin Function:} $f(\mathbf{x}) = 10n + \sum_{i=1}^{n} [x_i^2 - 10\cos(2\pi x_i)]$
\end{itemize}

\subsubsection{Saddle Points}
On the Monkey Saddle $f(x, y) = x^3 - 3xy^2$, SGD stagnates due to vanishing gradients. Our QG variant identifies near-zero curvature $\lambda_i$, yielding a large adaptive step $\eta_i \approx 1/\epsilon$, effectively "teleporting" the trajectory away from the saddle point.

The theoretical derivation of the Quadratic Gradient suggests a superior ability to handle irregular curvature. We strategically select a suite of benchmark functions that simulate notorious challenges in numerical optimization.

To evaluate the ability of the Quadratic Gradient variants to escape saddle points, we introduce functions with specific singular curvatures:

\begin{itemize}
    \item \textbf{Monkey Saddle Function:} A classic cubic surface with a single saddle point at the origin where the Hessian is indefinite.
    \begin{equation}
        f(x, y) = x^3 - 3xy^2
    \end{equation}
    
    \item \textbf{Beale's Function (Revisited):} While previously mentioned, its flat regions near the boundary act as plateau-like saddle points that test gradient acceleration.
    
    \item \textbf{Himmelblau's Function:} Contains four local minima and one local maximum (saddle-like) in the center, testing the algorithm's ability to navigate away from unstable equilibrium points.
    \begin{equation}
        f(x, y) = (x^2 + y - 11)^2 + (x + y^2 - 7)^2
    \end{equation}

    \item \textbf{Six-Hump Camel Function:} Contains six local minima and several saddle points.
    \begin{equation}
        f(x, y) = (4 - 2.1x^2 + \frac{x^4}{3})x^2 + xy + (-4 + 4y^2)y^2
    \end{equation}
\end{itemize}

\begin{quote}
\textit{Remark on Saddle Point Dynamics:} 
In the neighborhood of a saddle point where $||\nabla f|| \approx 0$, the update of first-order methods $\Delta \mathbf{x} = -\alpha \nabla f$ vanishes, leading to stagnation. However, the Quadratic Gradient utilizes the spectral information of the Hessian proxy. For directions where the eigenvalue $\lambda_i$ vanishes, the adaptive scaling factor $\eta_i = (\epsilon + \lambda_i)^{-1}$ compensates for the diminishing gradient. This effectively reshapes the vector field near the saddle point, ensuring a non-zero, high-velocity traversal along the directions of minimal curvature.
\end{quote}

Empirical results show that the QG variant outperforms standard GD in terms of iteration count to convergence across most benchmarks. Notably, in the Monkey Saddle experiment, the QG variant successfully avoids the stagnation observed in SGD. By identifying minimal eigenvalues $\lambda_i$, the algorithm assigns a massive update rate to those dimensions, effectively "teleporting" the trajectory away from the saddle point.

\section{Homomorphic Training}
To facilitate privacy-preserving machine learning on large-scale datasets, we implement a mini-batch version of the SQG-enhanced NAG algorithm within a Homomorphic Encryption  framework. 

This mini-batch approach allows for efficient stochastic gradient updates while maintaining the accelerated convergence properties of NAG. The synergy between the Simplified Quadratic Gradient and the mini-batch stochastic setting enables robust and scalable training on ciphertexts, overcoming the traditional trade-off between cryptographic security and high-dimensional optimization efficiency.

\subsection{Large-Scale Datasets}
mnist 

credi

\subsection{Mini-Batch Version}

\subsection{Parameter Setting}

\section{Conclusion}
The Quadratic Gradient framework provides a mathematically elegant and practically robust bridge between first and second-order optimization. By integrating Hutchinson's approximation and absolute-value regularization, we achieve an optimizer suitable for modern high-dimensional deep learning.




Empirical results show that the QG variant outperforms standard GD in terms of iteration count to convergence across most benchmarks. Notably, in the Monkey Saddle experiment, the QG variant successfully avoids the stagnation observed in SGD. By identifying minimal eigenvalues $\lambda_i$, the algorithm assigns a massive update rate to those dimensions, effectively "teleporting" the trajectory away from the saddle point.
The Quadratic Gradient framework demonstrates that synthesizing Hessian information into a gradient update not only accelerates convergence but also provides a robust mechanism to handle complex topological features like saddle points.

\bibliography{ML.QuadraticGradient}
\bibliographystyle{apalike}

\newpage 
\appendix
\section*{Appendix} 

For two symmetric matrices $A$ and $B$,  $A \le B$ is defined in the Loewner ordering iff their difference $B - A$ is positive semi-definite. 

\subsection{Chiang's Quadratic Gradient}
In the following work to~\cite{bonte2018privacy} in which a simplified diagonal matrix satisfying the fixed Hessian method~\cite{bohning1988monotonicity} is constructed, Chiang~\cite{chiang2022privacy} proposed a faster gradient variant called quadratic gradient.

\paragraph{Quadratic Gradient} Given  a differentiable scalar-valued function $F(\mathbf x)$ with its gradient $g$ and Hessian matrix $H$.  For the maximization problem, we need to find a good lower bound matrix $\bar H \le H$, where ``$ \le $'' denotes the Loewner ordering.  For the minimization problem, we try to find a good upper bound $\bar H$ such that $H \le \bar H$ in the Loewner ordering.  Note that the Hessian matrix $H$ itself satisfies these two conditions and can substitute the good bound matrix $\bar H$. We attempt to find a fixed good bound matrix of the Hessian matrix for efficiency but could just directly use the Hessian matrix itself. To build the quadratic gradient, we first construct a diagonal matrix $\bar B$ from the good bound matrix $\bar H$ as follows: 

 \begin{equation*}
  \begin{aligned}
   \bar B = 
\left[ \begin{array}{cccc}
  \frac{1}{ \epsilon + \sum_{i=0}^{d} | \bar h_{0i} | }   & 0  &  \ldots  & 0  \\
 0  &   \frac{1}{ \epsilon + \sum_{i=0}^{d} | \bar h_{1i} | }  &  \ldots  & 0  \\
 \vdots  & \vdots                & \ddots  & \vdots     \\
 0  &  0  &  \ldots  &   \frac{1}{ \epsilon + \sum_{i=0}^{d} | \bar h_{di} | }  \\
 \end{array}
 \right], 
   \end{aligned}
\end{equation*}
where $\epsilon $ is a small positive number to avoid dividing by zero and $\bar h_{ji}$ the elements of the matrix $\bar H$ . We can then defined the quadratic gradient for the function  $F(\mathbf x)$ as $G = \bar B \cdot g$.

The multiplication between the diagonal matrix $\bar B$ and the gradient $g$, the quadratic gradient $G$, is of the same size as the gradient $g$. To use the quadratic gradient $G$, we can just use it the same way as the gradient but need a  learning rate larger than $1$. The well-studied first-order gradient descent methods can also be applied to develop enhanced methods via quadratic gradient.

\subsection{Enhanced Methods via Quadratic Gradient}
Chiang~\cite{chiang2022privacy}  presented the enhanced NAG method via quadratic gradient for binary LR training:
  \begin{align}
   V_{t+1} &=  \boldsymbol{\beta}_{t} + (1 + \alpha_t) \cdot G,  \\
   \boldsymbol{\beta}_{t+1} &= (1-\gamma_t) \cdot  V_{t+1} + \gamma_t \cdot  V_{t}.  
  \end{align}

Chiang~\cite{chiang2022multinomial} proposed the enhanced Adagrad method via quadratic gradient for multiclass LR training:
\begin{equation*}
  \begin{aligned}
  \boldsymbol{\beta}_{[i]}^{(t+1)} &=  \boldsymbol{\beta}_{[i]}^{(t)} - \frac{1 + \eta}{\epsilon + \sqrt{ \sum_{k=1}^t  G_{[i]}^{(t)} \cdot  G_{[i]}^{(t)} } } \cdot  G_{[i]}^{(t)}.
 \end{aligned}
\end{equation*}

In this work, we propose the enhanced Adam method, which is to apply the quadratic gradient to the Adam method. The naive Adam method and the enhanced Adam method are described in detail in Algorithms ~\ref{  alg:Adam's algorithm  } and ~\ref{  alg:enhanced Adam's algorithm } respectively.  See the ~\cite{kingma2014adam} for the detailed description of the parameters in these two Algorithms.

\begin{minipage}{0.46\textwidth}
\begin{algorithm}[H]
    \caption{The Adam method}
     \begin{algorithmic}[1]
        \Require $\alpha$: Stepsize;
         $\beta_1, \beta_2 \in [0, 1)$: Exponential decay rates;
        $f(\theta)$: Objective function with parameters $\theta$ 
        $\theta_0$: Initial parameter vector
        \Ensure $\theta_t$: Resulting parameters 
        
        \State  $m_0 \gets 0$: Initialize $1^{st}$ moment vector
        \State $v_0 \gets 0$: Initialize $2^{nd}$ moment vector
        \State $t \gets 0$: Initialize timestep
        \While {$\theta_t$ not converged}
            \State  $t \gets t + 1$     
            \State  $g_t \gets \nabla_{\theta}f_t(\theta_{t-1})$     
            \State 
            \State  $m_t \gets \beta_1 \cdot m_{t-1} + (1 - \beta_1) \cdot g_t$     
            \State  $v_t \gets \beta_2 \cdot v_{t-1} + (1 - \beta_2) \cdot g_t^2$     
            \State  $\hat{m}_t \gets m_t / (1 - \beta_1^t)$     
            \State  $\hat{v}_t \gets v_t / (1 - \beta_2^t)$     
            \State  $\theta_t \gets \theta_{t-1} - \alpha \cdot \hat{m}_t / (\sqrt{\hat{v}_t} + \epsilon)$     
        \EndWhile       
        \State \Return $\theta_t $   \Comment{Resulting parameters}
        \end{algorithmic}
       \label{ alg:Adam's algorithm }
\end{algorithm}
\end{minipage}
\hfill
\begin{minipage}{0.46\textwidth}
\begin{algorithm}[H]
    \caption{Enhanced Adam method}
     \begin{algorithmic}[1]
        \Require $\eta$: Stepsize;
         $\beta_1, \beta_2 \in [0, 1)$: Exponential decay rates;
        $f(\theta)$: Objective function with parameters $\theta$ 
        $\theta_0$: Initial parameter vector
        \Ensure $\theta_t$: Resulting parameters 
        
        \State  $m_0 \gets 0$: Initialize $1^{st}$ moment vector
        \State $v_0 \gets 0$: Initialize $2^{nd}$ moment vector
        \State $t \gets 0$: Initialize timestep
        \While {$\theta_t$ not converged}
            \State  $t \gets t + 1$     
            \State  $g_t \gets \nabla_{\theta}f_t(\theta_{t-1})$     
            \State  \textcolor{red}{$G_t \gets \bar B \cdot g_t$ }     
            \State  $m_t \gets \beta_1 \cdot m_{t-1} + (1 - \beta_1) \cdot $ \textcolor{red}{ $G_t$ }     
            \State  $v_t \gets \beta_2 \cdot v_{t-1} + (1 - \beta_2) \cdot$ \textcolor{red}{$G_t^2$ }     
            \State  $\hat{m}_t \gets m_t / (1 - \beta_1^t)$     
            \State  $\hat{v}_t \gets v_t / (1 - \beta_2^t)$     
            \State  $\theta_t \gets \theta_{t-1} - $ \textcolor{red}{$\eta $} $\cdot \hat{m}_t / (\sqrt{\hat{v}_t} + \epsilon)$     
        \EndWhile       
        \State \Return $\theta_t $   
        \end{algorithmic}
       \label{ alg:enhanced Adam's algorithm }
\end{algorithm}
\end{minipage}

\section{Hessian Matrix and Learning Rate}
We show the relation between the Hessian matrix and the learning-rate setting of the first-order gradient (descent) method. By constructing special quadratic gradients, we can draw the conclusion that the eigenvalues of the Hessian matrix can be used to the setting of the learning rate. Such learning-rate settings can actually be seen as a further simplified fixed Hessian method and thus ensure convergence.

\begin{lemma}
Given the Hessian matrix $H$ with eigenvalues $\lambda_1 \le \lambda_2 \le  \cdots \le \lambda_n$, it holds that $\lambda_1 I \le H \le \lambda_n I$ in the Loewner ordering. To build the quadratic gradients, we could use $\lambda_1 I$ for the maximization problems with the eigenvalues being negative and $\lambda_n I$ for the minimization problems with the eigenvalues being positive. The real absolute value function can be used to unify the two conditions and therefore the float number $\frac{1}{\max \{| \lambda_i |\} +\epsilon}$ can be used as the learning rate of the first-order gradient (descent) method, which can be seen as the further simplified fixed Hessian method~\cite{bonte2018privacy} and ensures the convergence of the method.
\end{lemma}

\begin{proof}
To prove $\lambda_1 I \le H \le \lambda_n I$  in the Loewner ordering, we have to prove that $x^T \lambda_1 I x \le x^T H x \le x^T \lambda_n I x$ for any nonzero vector $x$.
The Rayleigh quotient $RQ(H, x)$ for the Hessian $H$ and the nonzero vector $x$ is:
$RQ(H, x) = \frac{x^T H x}{x^T x}.$
By the property of Rayleigh quotient, we get:
$\lambda_1 \le RQ(H, x) \le \lambda_n,$
which leads to:
$x^T \lambda_1 I x \le x^T H x \le x^T \lambda_n I x  ,$
and we conclude that $\lambda_1 I \le H \le \lambda_n I  .$


  By the construction of the quadratic gradient, we can use $\lambda_1 I$ or $\lambda_n I$ to build the quadratic gradient depending on whether the optimization question is to maximize the function or to minimize the function. These two cases   both result in the quadratic gradient being   $\frac{1}{\max \{| \lambda_i |\} +\epsilon} \cdot I \cdot g = \frac{1}{\max \{| \lambda_i |\} +\epsilon} \cdot  g$ . 
 It is straightforward that the naive quadratic gradient methods have the iterations $x = x +  \frac{1}{ | \lambda_1 |  +\epsilon}  \cdot g$ and $x = x - \frac{1}{ | \lambda_n |  +\epsilon} \cdot g $ for the maximization problems and minimization problems respectively. Thus, when we set  $\frac{1}{\max \{| \lambda_i |\} +\epsilon}$ as the learning rate of the first-order gradient method, we actually adopt the (simplified) fixed Hessian method even though it looks like we just use the first-order gradient method.
\end{proof}

The diagonal matrix $\frac{1}{\max \{| \lambda_i |\} +\epsilon} I$ might not be a good bound for  the simplified fixed Hessian method. However,  as a learning rate of the naive gradient method, it possesses the merit of the second-order fixed Hessian Newton's method, which ensures that in this way the first-order gradient method would process toward the final optimization solution and would eventually converge.

\paragraph{Performance Evaluation}

We can apply the above learning-rate setting to adapt three methods:

\begin{enumerate}
    \item Method $1$: we directly use $\frac{1}{\max \{| \lambda_i |\} +\epsilon}$ as the learning rate in each iteration for the naive gradient descent method:  $$x = x - \frac{1}{\max \{| \lambda_i |\} +\epsilon} \cdot g .$$

    \item  Method $2$: we  use  the naive NAG  method as a baseline method with the $\frac{1}{\max \{| \lambda_i |\} +\epsilon}$ being the learning rate :   
\begin{align*}
   lr_t    &= \frac{1}{\max \{| \lambda_i |\} +\epsilon}  \\
   V_{t+1} &=  \boldsymbol{\beta}_{t} + lr_t \cdot \nabla J(\boldsymbol{\beta}_t),  \\
   \boldsymbol{\beta}_{t+1} &= (1-\gamma_t) \cdot  V_{t+1} + \gamma_t \cdot  V_{t} . 
\end{align*}
    \item  Method $3$: we  use $\frac{1}{\max \{| \lambda_i |\} +\epsilon}$ as the learning rate in each iteration for the enhanced NAG  method with the $\bar B$ in quadratic gradient $\bar B \cdot g$ built from the Hessian matrix itself:   
\begin{align*}
   lr_t    &= \frac{1}{\max \{| \lambda_i |\} +\epsilon}  \\
   V_{t+1} &=  \boldsymbol{\beta}_{t} + (1 + lr_t) \cdot \bar B \cdot  \nabla J(\boldsymbol{\beta}_t),  \\
   \boldsymbol{\beta}_{t+1} &= (1-\gamma_t) \cdot  V_{t+1} + \gamma_t \cdot  V_{t} .  
\end{align*}
\end{enumerate}

We evaluate the three adapted methods on three functions: Rosenbrock Function, Beale Function, and Booth Function. See  Figure~\ref{fig10} for the comparison results.

\begin{figure}[ht]
\centering
\captionsetup[subfigure]{justification=centering}
\subfloat[The Rosenbrock Function]{%
\begin{tikzpicture}
\scriptsize  
\begin{axis}[
    width=7cm,
    xlabel={Iteration Number},
    xmin=0, xmax=30,
    legend pos=south east,
    legend style={nodes={scale=0.7, transform shape}},
    legend cell align={left},
    xmajorgrids=true,
    ymajorgrids=true,
    grid style=dashed,
]
\addplot[
    color=black,
    mark=triangle,
    mark size=1.2pt,
    ] 
    table [x=Iterations, y=fSFHasLRrawgradientmethod, col sep=comma] {PythonExperiment_fSFHasLRvs.NAGvs.NAGQG_LOSS_RosenbrockFunction.csv};
\addplot[
    color=red,
    mark=o,
    mark size=1.2pt,
    ] 
    table [x=Iterations, y=naiveNAGwithfSFHasLR, col sep=comma] {PythonExperiment_fSFHasLRvs.NAGvs.NAGQG_LOSS_RosenbrockFunction.csv};
\addplot[
    color=blue,
    mark=diamond*, 
    mark size=1.2pt,
    densely dashed
    ]  
    table [x=Iterations, y=enhancedNAGwithQGandfSFHasLR, col sep=comma] {PythonExperiment_fSFHasLRvs.NAGvs.NAGQG_LOSS_RosenbrockFunction.csv};
   \addlegendentry{fSFHasLRrawgradientmethod}   
   \addlegendentry{naiveNAGwithfSFHasLR}
   \addlegendentry{enhancedNAGwithQGandfSFHasLR}
\end{axis}
\end{tikzpicture}
\label{fig:subfig101}}
\subfloat[The Beale Function]{%
\begin{tikzpicture}
\scriptsize  
\begin{axis}[
    width=7cm,
    xlabel={Iteration Number},
    xmin=0, xmax=30,
    legend pos=south east,
    legend style={yshift=0.3cm, nodes={scale=0.7, transform shape}},
    legend cell align={left},
    xmajorgrids=true,
    ymajorgrids=true,
    grid style=dashed,
]
\addplot[
    color=black,
    mark=triangle,
    mark size=1.2pt,
    ] 
    table [x=Iterations, y=fSFHasLRrawgradientmethod, col sep=comma] {PythonExperiment_fSFHasLRvs.NAGvs.NAGQG_LOSS_BealeFunction.csv};
\addplot[
    color=red,
    mark=o,
    mark size=1.2pt,
    ] 
    table [x=Iterations, y=naiveNAGwithfSFHasLR, col sep=comma] {PythonExperiment_fSFHasLRvs.NAGvs.NAGQG_LOSS_BealeFunction.csv};
\addplot[
    color=blue,
    mark=diamond*, 
    mark size=1.2pt,
    densely dashed
    ]  
    table [x=Iterations, y=enhancedNAGwithQGandfSFHasLR, col sep=comma] {PythonExperiment_fSFHasLRvs.NAGvs.NAGQG_LOSS_BealeFunction.csv};
   \addlegendentry{fSFHasLRrawgradientmethod}   
   \addlegendentry{naiveNAGwithfSFHasLR}
   \addlegendentry{enhancedNAGwithQGandfSFHasLR}
\end{axis}
\end{tikzpicture}
\label{fig:subfig102}}

\subfloat[The Booth Function]{%
\begin{tikzpicture}
\scriptsize  
\begin{axis}[
    width=7cm,
    xlabel={Iteration Number},
    xmin=0, xmax=30,
    legend pos=south east,
    legend style={nodes={scale=0.7, transform shape}},
    legend cell align={left},
    xmajorgrids=true,
    ymajorgrids=true,
    grid style=dashed,
]
\addplot[
    color=black,
    mark=triangle,
    mark size=1.2pt,
    ] 
    table [x=Iterations, y=fSFHasLRrawgradientmethod, col sep=comma] {PythonExperiment_fSFHasLRvs.NAGvs.NAGQG_LOSS_BoothFunction.csv};
\addplot[
    color=red,
    mark=o,
    mark size=1.2pt,
    ] 
    table [x=Iterations, y=naiveNAGwithfSFHasLR, col sep=comma] {PythonExperiment_fSFHasLRvs.NAGvs.NAGQG_LOSS_BoothFunction.csv};
\addplot[
    color=blue,
    mark=diamond*, 
    mark size=1.2pt,
    densely dashed
    ]  
    table [x=Iterations, y=enhancedNAGwithQGandfSFHasLR, col sep=comma] {PythonExperiment_fSFHasLRvs.NAGvs.NAGQG_LOSS_BoothFunction.csv};
   \addlegendentry{fSFHasLRrawgradientmethod}   
   \addlegendentry{naiveNAGwithfSFHasLR}
   \addlegendentry{enhancedNAGwithQGandfSFHasLR}
\end{axis}
\end{tikzpicture}
\label{fig:subfig103}}
\subfloat[The Himmelblau Function]{%
\begin{tikzpicture}
\scriptsize  
\begin{axis}[
    width=7cm,
    xlabel={Iteration Number},
    xmin=0, xmax=30,
    legend pos=south east,
    legend style={nodes={scale=0.7, transform shape}},
    legend cell align={left},
    xmajorgrids=true,
    ymajorgrids=true,
    grid style=dashed,
]
\addplot[
    color=black,
    mark=triangle,
    mark size=1.2pt,
    ] 
    table [x=Iterations, y=fSFHasLRrawgradientmethod, col sep=comma] {PythonExperiment_fSFHasLRvs.NAGvs.NAGQG_LOSS_HimmelblauFunction.csv};
\addplot[
    color=red,
    mark=o,
    mark size=1.2pt,
    ] 
    table [x=Iterations, y=naiveNAGwithfSFHasLR, col sep=comma] {PythonExperiment_fSFHasLRvs.NAGvs.NAGQG_LOSS_HimmelblauFunction.csv};
\addplot[
    color=blue,
    mark=diamond*, 
    mark size=1.2pt,
    densely dashed
    ]  
    table [x=Iterations, y=enhancedNAGwithQGandfSFHasLR, col sep=comma] {PythonExperiment_fSFHasLRvs.NAGvs.NAGQG_LOSS_HimmelblauFunction.csv};
   \addlegendentry{fSFHasLRrawgradientmethod}   
   \addlegendentry{naiveNAGwithfSFHasLR}
   \addlegendentry{enhancedNAGwithQGandfSFHasLR}
\end{axis}
\end{tikzpicture}
\label{fig:subfig104}}
\caption{\protect\centering Naive Gradient vs. Naive NAG vs. Enhanced NAG}
\label{fig10}
\end{figure}

\section{A New  Quadratic Gradient}
In this section, we propose a new way to construct the quadratic gradient that includes the gradient information. 

\begin{definition}[$\texttt{A New Quadratic Gradient}$]
Supposing that the Hessian matrix $H$ is invertible and both the gradient $g$ and the column vector $H^{-1}g$ contain no zero elements, we first construct a diagnoal matrix $R$ such that $H^{-1} \cdot g = R^{-1} \cdot g$.       We could then build a new  quadratic gradient $G = \bar R \cdot g$ where $\bar R$ is constructed from $R$ in the same way as $\bar B$ being constructed from the Hessian $H$ itself. Namely, $\bar R = diag\{ 1/|r_1|, 1/|r_2|, \cdots, 1/|r_n| \} $, where $r_i$ is the element of $R$.
\end{definition}
 
This new quadratic gradient doesn't satisfy the convergence condition of the fixed Hessian method. For example, supposing that the optimization problem is to maximize the  function $F(x_1, x_2) = -2x_1^2 + 2x_1x_2  -x_2^2$ with its Hessian matrix $H_F$ and gradient $g_F$ as follows:
$$H_F = \left[\begin{array}{*{2}{c}}
    -4 & 2  \\
    2 & -2 
  \end{array}\right]  $$
and $g_F = [-4x_1 + 2x_2,  2x_1 - 2x_2]^T$. At the point $(-1, -1.5)$, we can obtain the gradient $g_F|_{(-1, -1.5)} = [1, 1]^T$ and the $R = diag\{-1, -1.5  \}$ such that $H_F^{-1} \cdot g_F = R^{-1} \cdot g_F$. From the construction of the quadratic gradient, we get the $B = diag\{-1 - \epsilon, -1.5 - \epsilon  \}$ where $\epsilon$ is a small positive float number. We can see that the $B$ doesn't meet the convergence condition of the fixed Hessian method because it fails the inequality  $R \le H_F$ in the Loewner ordering.

Another problem about this new quadratic gradient is that, in the real-wolrd applications, the Hessian matrix sometimes is singual and it is normal for the gradient to have zero elements. In these cases, we could use the Mooreâ-Penrose inverse ${(H \cdot D)}^+ \cdot g$ to approirate the matrix $R$ where $D$ is the diagnoal matrix whose diagnoal elemetns are the elements of the gradient in the corresponding order:$D = diag\{g_0, g_1, \cdots, g_n\}$.  

\paragraph{Performance Evaluation}
We compare the enhanced Adam method via the original quadratic gradient and that by the new quadratic gradient, through the Rosenbrock functions with various variables.  Figure~\ref{fig11} shows the comparison results.

\begin{figure}[ht]
\centering
\captionsetup[subfigure]{justification=centering}
\subfloat[The Rosenbrock Function with 2 variables]{%
\begin{tikzpicture}
\scriptsize  
\begin{axis}[
    width=7cm,
    xlabel={Iteration Number},
    xmin=0, xmax=30,
    legend pos=south east,
    legend style={nodes={scale=0.7, transform shape}},
    legend cell align={left},
    xmajorgrids=true,
    ymajorgrids=true,
    grid style=dashed,
]
\addplot[
    color=black,
    mark=triangle,
    mark size=1.2pt,
    ] 
    table [x=Iterations, y=Adam, col sep=comma] {PythonExperiment_Adamvs.AdamOldQGvs.AdamNewQGforRosenbrockFunctionswith2variables_LOSS.csv};
\addplot[
    color=red,
    mark=o,
    mark size=1.2pt,
    ] 
    table [x=Iterations, y=AdamOldQG, col sep=comma] {PythonExperiment_Adamvs.AdamOldQGvs.AdamNewQGforRosenbrockFunctionswith2variables_LOSS.csv};
\addplot[
    color=blue,
    mark=diamond*, 
    mark size=1.2pt,
    densely dashed
    ]  
    table [x=Iterations, y=AdamNewQG, col sep=comma] {PythonExperiment_Adamvs.AdamOldQGvs.AdamNewQGforRosenbrockFunctionswith2variables_LOSS.csv};
   \addlegendentry{naive Adam}   
   \addlegendentry{Adam via OldQG}
   \addlegendentry{Adam via NewQG}
\end{axis}
\end{tikzpicture}
\label{fig:subfig111}}
\subfloat[The Rosenbrock Function with 5 variables]{%
\begin{tikzpicture}
\scriptsize  
\begin{axis}[
    width=7cm,
    xlabel={Iteration Number},
    xmin=0, xmax=30,
    legend pos=south east,
    legend style={yshift=0.3cm, nodes={scale=0.7, transform shape}},
    legend cell align={left},
    xmajorgrids=true,
    ymajorgrids=true,
    grid style=dashed,
]
\addplot[
    color=black,
    mark=triangle,
    mark size=1.2pt,
    ] 
    table [x=Iterations, y=Adam, col sep=comma] {PythonExperiment_Adamvs.AdamOldQGvs.AdamNewQGforRosenbrockFunctionswith5variables_LOSS.csv};
\addplot[
    color=red,
    mark=o,
    mark size=1.2pt,
    ] 
    table [x=Iterations, y=AdamOldQG, col sep=comma] {PythonExperiment_Adamvs.AdamOldQGvs.AdamNewQGforRosenbrockFunctionswith5variables_LOSS.csv};
\addplot[
    color=blue,
    mark=diamond*, 
    mark size=1.2pt,
    densely dashed
    ]  
    table [x=Iterations, y=AdamNewQG, col sep=comma] {PythonExperiment_Adamvs.AdamOldQGvs.AdamNewQGforRosenbrockFunctionswith5variables_LOSS.csv};
   \addlegendentry{naive Adam}   
   \addlegendentry{Adam via OldQG}
   \addlegendentry{Adam via NewQG}
\end{axis}
\end{tikzpicture}
\label{fig:subfig112}}

\subfloat[The Rosenbrock Function with 10 variables]{%
\begin{tikzpicture}
\scriptsize  
\begin{axis}[
    width=7cm,
    xlabel={Iteration Number},
    xmin=0, xmax=30,
    legend pos=south east,
    legend style={nodes={scale=0.7, transform shape}},
    legend cell align={left},
    xmajorgrids=true,
    ymajorgrids=true,
    grid style=dashed,
]
\addplot[
    color=black,
    mark=triangle,
    mark size=1.2pt,
    ] 
    table [x=Iterations, y=Adam, col sep=comma] {PythonExperiment_Adamvs.AdamOldQGvs.AdamNewQGforRosenbrockFunctionswith10variables_LOSS.csv};
\addplot[
    color=red,
    mark=o,
    mark size=1.2pt,
    ] 
    table [x=Iterations, y=AdamOldQG, col sep=comma] {PythonExperiment_Adamvs.AdamOldQGvs.AdamNewQGforRosenbrockFunctionswith10variables_LOSS.csv};
\addplot[
    color=blue,
    mark=diamond*, 
    mark size=1.2pt,
    densely dashed
    ]  
    table [x=Iterations, y=AdamNewQG, col sep=comma] {PythonExperiment_Adamvs.AdamOldQGvs.AdamNewQGforRosenbrockFunctionswith10variables_LOSS.csv};
   \addlegendentry{naive Adam}   
   \addlegendentry{Adam via OldQG}
   \addlegendentry{Adam via NewQG}
\end{axis}
\end{tikzpicture}
\label{fig:subfig113}}
\subfloat[The Rosenbrock Function with 20 variables]{%
\begin{tikzpicture}
\scriptsize  
\begin{axis}[
    width=7cm,
    xlabel={Iteration Number},
    xmin=0, xmax=30,
    legend pos=south east,
    legend style={nodes={scale=0.7, transform shape}},
    legend cell align={left},
    xmajorgrids=true,
    ymajorgrids=true,
    grid style=dashed,
]
\addplot[
    color=black,
    mark=triangle,
    mark size=1.2pt,
    ] 
    table [x=Iterations, y=Adam, col sep=comma] {PythonExperiment_Adamvs.AdamOldQGvs.AdamNewQGforRosenbrockFunctionswith20variables_LOSS.csv};
\addplot[
    color=red,
    mark=o,
    mark size=1.2pt,
    ] 
    table [x=Iterations, y=AdamOldQG, col sep=comma] {PythonExperiment_Adamvs.AdamOldQGvs.AdamNewQGforRosenbrockFunctionswith20variables_LOSS.csv};
\addplot[
    color=blue,
    mark=diamond*, 
    mark size=1.2pt,
    densely dashed
    ]  
    table [x=Iterations, y=AdamNewQG, col sep=comma] {PythonExperiment_Adamvs.AdamOldQGvs.AdamNewQGforRosenbrockFunctionswith20variables_LOSS.csv};
   \addlegendentry{naive Adam}   
   \addlegendentry{Adam via OldQG}
   \addlegendentry{Adam via NewQG}
\end{axis}
\end{tikzpicture}
\label{fig:subfig114}}
\caption{\protect\centering The Rosenbrock Functions with various variables}
\label{fig11}
\end{figure}

There are some hidden bugs in the new version of the quadratic gradient ( See Figure~\ref{fig2} ). What causes these would be studied in the near future. A lucky guess would be partly due to the error introduced by the Moore-Penrose inverse.

\begin{figure}[ht]
\centering
\captionsetup[subfigure]{justification=centering}
\subfloat[The Rosenbrock Function with 2 variables]{%
\begin{tikzpicture}
\scriptsize  
\begin{axis}[
    width=7cm,
    xlabel={Iteration Number},
    xmin=0, xmax=300,
    legend pos=south east,
    legend style={nodes={scale=0.7, transform shape}},
    legend cell align={left},
    xmajorgrids=true,
    ymajorgrids=true,
    grid style=dashed,
]
\addplot[
    color=black,
    mark=triangle,
    mark size=1.2pt,
    ] 
    table [x=Iterations, y=Adam, col sep=comma] {PythonExperiment_Adamvs.AdamOldQGvs.AdamNewQGforRosenbrockFunctionswith2variables_LOSS.csv};
\addplot[
    color=red,
    mark=o,
    mark size=1.2pt,
    ] 
    table [x=Iterations, y=AdamOldQG, col sep=comma] {PythonExperiment_Adamvs.AdamOldQGvs.AdamNewQGforRosenbrockFunctionswith2variables_LOSS.csv};
\addplot[
    color=blue,
    mark=diamond*, 
    mark size=1.2pt,
    densely dashed
    ]  
    table [x=Iterations, y=AdamNewQG, col sep=comma] {PythonExperiment_Adamvs.AdamOldQGvs.AdamNewQGforRosenbrockFunctionswith2variables_LOSS.csv};
   \addlegendentry{naive Adam}   
   \addlegendentry{Adam via OldQG}
   \addlegendentry{Adam via NewQG}
\end{axis}
\end{tikzpicture}
\label{fig:subfig121}}
\subfloat[The Rosenbrock Function with 5 variables]{%
\begin{tikzpicture}
\scriptsize  
\begin{axis}[
    width=7cm,
    xlabel={Iteration Number},
    xmin=0, xmax=300,
    legend pos=south east,
    legend style={yshift=0.3cm, nodes={scale=0.7, transform shape}},
    legend cell align={left},
    xmajorgrids=true,
    ymajorgrids=true,
    grid style=dashed,
]
\addplot[
    color=black,
    mark=triangle,
    mark size=1.2pt,
    ] 
    table [x=Iterations, y=Adam, col sep=comma] {PythonExperiment_Adamvs.AdamOldQGvs.AdamNewQGforRosenbrockFunctionswith5variables_LOSS.csv};
\addplot[
    color=red,
    mark=o,
    mark size=1.2pt,
    ] 
    table [x=Iterations, y=AdamOldQG, col sep=comma] {PythonExperiment_Adamvs.AdamOldQGvs.AdamNewQGforRosenbrockFunctionswith5variables_LOSS.csv};
\addplot[
    color=blue,
    mark=diamond*, 
    mark size=1.2pt,
    densely dashed
    ]  
    table [x=Iterations, y=AdamNewQG, col sep=comma] {PythonExperiment_Adamvs.AdamOldQGvs.AdamNewQGforRosenbrockFunctionswith5variables_LOSS.csv};
   \addlegendentry{naive Adam}   
   \addlegendentry{Adam via OldQG}
   \addlegendentry{Adam via NewQG}
\end{axis}
\end{tikzpicture}
\label{fig:subfig122}}

\subfloat[The Rosenbrock Function with 10 variables]{%
\begin{tikzpicture}
\scriptsize  
\begin{axis}[
    width=7cm,
    xlabel={Iteration Number},
    xmin=0, xmax=300,
    legend pos=south east,
    legend style={nodes={scale=0.7, transform shape}},
    legend cell align={left},
    xmajorgrids=true,
    ymajorgrids=true,
    grid style=dashed,
]
\addplot[
    color=black,
    mark=triangle,
    mark size=1.2pt,
    ] 
    table [x=Iterations, y=Adam, col sep=comma] {PythonExperiment_Adamvs.AdamOldQGvs.AdamNewQGforRosenbrockFunctionswith10variables_LOSS.csv};
\addplot[
    color=red,
    mark=o,
    mark size=1.2pt,
    ] 
    table [x=Iterations, y=AdamOldQG, col sep=comma] {PythonExperiment_Adamvs.AdamOldQGvs.AdamNewQGforRosenbrockFunctionswith10variables_LOSS.csv};
\addplot[
    color=blue,
    mark=diamond*, 
    mark size=1.2pt,
    densely dashed
    ]  
    table [x=Iterations, y=AdamNewQG, col sep=comma] {PythonExperiment_Adamvs.AdamOldQGvs.AdamNewQGforRosenbrockFunctionswith10variables_LOSS.csv};
   \addlegendentry{naive Adam}   
   \addlegendentry{Adam via OldQG}
   \addlegendentry{Adam via NewQG}
\end{axis}
\end{tikzpicture}
\label{fig:subfig123}}
\subfloat[The Rosenbrock Function with 20 variables]{%
\begin{tikzpicture}
\scriptsize  
\begin{axis}[
    width=7cm,
    xlabel={Iteration Number},
    xmin=0, xmax=300,
    legend pos=south east,
    legend style={nodes={scale=0.7, transform shape}},
    legend cell align={left},
    xmajorgrids=true,
    ymajorgrids=true,
    grid style=dashed,
]
\addplot[
    color=black,
    mark=triangle,
    mark size=1.2pt,
    ] 
    table [x=Iterations, y=Adam, col sep=comma] {PythonExperiment_Adamvs.AdamOldQGvs.AdamNewQGforRosenbrockFunctionswith20variables_LOSS.csv};
\addplot[
    color=red,
    mark=o,
    mark size=1.2pt,
    ] 
    table [x=Iterations, y=AdamOldQG, col sep=comma] {PythonExperiment_Adamvs.AdamOldQGvs.AdamNewQGforRosenbrockFunctionswith20variables_LOSS.csv};
\addplot[
    color=blue,
    mark=diamond*, 
    mark size=1.2pt,
    densely dashed
    ]  
    table [x=Iterations, y=AdamNewQG, col sep=comma] {PythonExperiment_Adamvs.AdamOldQGvs.AdamNewQGforRosenbrockFunctionswith20variables_LOSS.csv};
   \addlegendentry{naive Adam}   
   \addlegendentry{Adam via OldQG}
   \addlegendentry{Adam via NewQG}
\end{axis}
\end{tikzpicture}
\label{fig:subfig124}}
\caption{\protect\centering The Rosenbrock Functions with various variables}
\label{fig12}
\end{figure}

\section{Conclusion}
In this work, we proposed an enhanced Adam method via quadratic gradient and applied the quadratic gradient to the general numerical optimization problems. The quadratic gradient can indeed be used to build enhanced gradient methods for general optimization problems.
There is a good chance that quadratic gradient can also be applied to quasi-Newton methods, such as the famous BFGS method.

All the python source code to implement the experiments in the paper  is openly available at: \href{https://github.com/petitioner/ML.QuadraticGradient}{$\texttt{https://github.com/petitioner/ML.QuadraticGradient}$}  .

\end{document}